\documentclass[11pt]{amsart}%
\usepackage{ae,aecompl}
\usepackage{chngcntr}
\usepackage{graphicx}
\usepackage[all,cmtip]{xy}
\usepackage{amsmath, amscd}
\usepackage{amsthm}
\usepackage{amssymb}
\usepackage{amsfonts}
\usepackage{qsymbols}
\usepackage{latexsym}
\usepackage{mathrsfs}
\usepackage{cite}
\usepackage{color}
\usepackage{url}
\usepackage{enumerate}
\usepackage[utf8]{inputenc}
\usepackage[T1]{fontenc}
\usepackage{esint}
\usepackage{mathtools}
\usepackage{bbm}
\usepackage{verbatim}
\usepackage[colorlinks, citecolor = blue, pagebackref]{hyperref}
\usepackage[colorinlistoftodos,prependcaption,textsize=small]{todonotes}
\usepackage{accents}
\usepackage{amsmath}%
\providecommand{\U}[1]{\protect\rule{.1in}{.1in}}
\allowdisplaybreaks
\newcommand {\N}{{{\mathbb N}}}

\newcommand {\R}{{\mathbb R}}

\newcommand {\T}{{\mathbb T}}

\newcommand {\vanish}[1]{\relax}

\newtheorem{theorem}{Theorem}[section]
\newtheorem{lemma}[theorem]{Lemma}
\newtheorem{proposition}[theorem]{Proposition}
\newtheorem{corollary}[theorem]{Corollary}

\newtheorem{op}[theorem]{Open Problem}
\theoremstyle{definition}
\newtheorem{definition}[theorem]{Definition}
\newtheorem{remark}[theorem]{Remark}

\numberwithin{equation}{section}
\protected
\def\ignorethis#1\endignorethis{}
\let\endignorethis\relax

\begin{document}
\title[Bourgain--Brezis--Mironescu--Maz'ya--Shaposhnikova limit formulae]{Bourgain--Brezis--Mironescu--Maz'ya--Shaposhnikova limit formulae for fractional
Sobolev spaces via interpolation and extrapolation}
\author{Oscar Dom\'inguez}
\address{O. Dom\'inguez, Universit\'e Lyon 1, Institut Camille Jordan, 43 blvd. du 11 novembre 1918, F-69622
Villeurbanne cedex, France; Departamento de An\'alisis Matem\'atico y Matem\'atica Aplicada, Facultad de Matem\'aticas, Universidad Complutense de Madrid, Plaza de Ciencias 3, 28040 Madrid, Spain}
\email{dominguez@math.univ-lyon1.fr}
\email{oscar.dominguez@ucm.es}

\author{Mario Milman}
\address{M. Milman, Instituto Argentino de Matematica\\
Buenos Aires\\
Argentina}
\email{mario.milman@icloud.com}
\urladdr{https://sites.google.com/site/mariomilman/}

\begin{abstract}
The real interpolation spaces between $L^{p}({\mathbb{R}}^{n})$ and $\dot
{H}^{t,p}({\mathbb{R}}^{n})$ (resp. $H^{t,p}({\mathbb{R}}^{n})$), \, $t>0,$ are
characterized in terms of fractional moduli of smoothness, and the underlying
seminorms are shown to be \textquotedblleft the correct\textquotedblright%
\ fractional generalization of the classical Gagliardo seminorms. This is
confirmed by the fact that, using the new spaces combined with interpolation
and extrapolation methods, we are able to extend the
Bourgain--Brezis--Mironescu--Maz'ya--Shaposhnikova limit formulae, as well as the
Bourgain--Brezis--Mironescu convergence theorem, to fractional Sobolev spaces.
On the other hand, we disprove a conjecture of \cite{Braz} suggesting
fractional convergence results given in terms of classical Gagliardo
seminorms. We also solve a problem proposed in \cite{Braz} concerning sharp
forms of the fractional Sobolev embedding.

\end{abstract}
\maketitle
%\tableofcontents

\section{Preamble}

\subsection{The Bourgain--Brezis--Mironescu--Maz'ya--Shaposhnikova limit formulae}

In their celebrated paper \cite{BBM}, Bourgain--Brezis--Mironescu studied limits
of Gagliardo semi-norms defined by\footnote{We assume that $p>1,$ since, as in
\cite{Braz}, our main focus will be the fractional Sobolev spaces defined via
$(-\Delta)^{\frac{t}{2}}.$ $\ $The case $p=1$ requires a separate treatment.}%
\begin{equation}
\Vert f\Vert_{\dot{W}^{s,p}({\mathbb{R}}^{n})}=\left(  \int_{{\mathbb{R}}%
^{n}}\int_{{\mathbb{R}}^{n}}\frac{\left\vert f(x)-f(y)\right\vert ^{p}%
}{\left\vert x-y\right\vert ^{n+sp}}dxdy\right)^{1/p},\qquad s\in(0,1),\quad p\in
(1,\infty), \label{gag1}%
\end{equation}
and showed that while%

\begin{equation}
\lim_{s\rightarrow1^{-}}\Vert f\Vert_{\dot{W}^{s,p}({\mathbb{R}}^{n})}<\infty\iff
f\quad\text{is constant}, \label{Constancy}%
\end{equation}
the Gagliardo seminorms can be normalized so that the $L^{p}$ norm of the
gradient can be recovered: For $f\in W_{p}^{1}({\mathbb{R}}^{n})$, it holds
\begin{equation}
\lim_{s\rightarrow1^{-}}(1-s)^{\frac{1}{p}}\Vert f\Vert_{\dot{W}^{s,p}%
({\mathbb{R}}^{n})}=C(p,n)\Vert\nabla f\Vert_{L^{p}({\mathbb{R}}^{n})},
\label{BBM}%
\end{equation}
where $C(p,n)$ is a constant that depends only on $p$ and $n,$ and can be
computed explicitly.

Later, Maz'ya--Shaposhnikova \cite{Mazya} used a different normalization to
deal with the case where the (smoothness) parameter $s$ tends to zero, and
proved that
\begin{equation}
\lim_{s\rightarrow0^{+}}s^{\frac{1}{p}}\Vert f\Vert_{\dot{W}^{s,p}({\mathbb{R}%
}^{n})}=c(p,n)\Vert f\Vert_{L^{p}({\mathbb{R}}^{n})}. \label{MS}%
\end{equation}

These results have sparked intense research interest and have been used and
extended in many different directions (for recent related material,
applications and references cf. \cite{Braz}, \cite{BN}, \cite{BGT}, \cite{cejas}, \cite{GT}, \cite{polam}, \cite{Xu}). The
original proofs were somewhat complicated, and the process of finding the
correct normalizations involved ad-hoc considerations. For example, the
correct normalization for (\ref{MS}) was found using the sharp form of the
Sobolev embedding and, conversely, the limit result (\ref{BBM}) suggested the
sharp Fractional Sobolev inequality obtained in \cite{BBM1}.

In \cite{Milman}, \cite{KaradzhovMilmanXiao}, it was observed that these
limiting results can be naturally understood in the more general context of
scales of interpolation spaces. Given a pair of compatible spaces
$(X_{0},X_{1}),$ interpolation methods (cf. \cite{BerghLofstrom}) give rise to
scales of spaces, in particular, the scale $(X_{0},X_{1})_{s,q},$
$s\in(0,1),q\in\lbrack1,\infty],$ produced by the real method\footnote{See
Section \ref{SectionMethod} for background information on interpolation and
extrapolation.}, will play an important r\^{o}le in this paper. We
colloquially refer to $X_{0}$ and $X_{1}$ as \textquotedblleft end point
spaces\textquotedblright,\ and we think that, as the parameter $s$ varies from
$0$ to $1,$ we are moving from one end point space in the scale, to the other.
In this setting, the limit theorems take the form of a continuity result that
can be informally interpreted as saying that \textquotedblleft with
appropriate normalizations we can recover the end point
spaces\textquotedblright\ (cf. Theorem \ref{teomarkao}),
\begin{equation}
\lim_{s\rightarrow0^{+}} c_{s,q} \left\Vert f\right\Vert _{(X_{0},X_{1}%
)_{s,q}}=\sup_{t>0}K(t,f;X_{0},X_{1}); \quad \text{ }\lim_{s\rightarrow1^{-}}%
c_{s,q}\left\Vert f\right\Vert _{(X_{0},X_{1})_{s,q}}=\sup_{t>0}%
\frac{K(t,f;X_{0},X_{1})}{t}. \label{m1}%
\end{equation}
For many pairs of spaces $(X_{0},X_{1}),$ we actually have estimates: For
$f\in X_{0}\cap X_{1},$\footnote{Throughout the paper, the symbol $f \lesssim g$ indicates the existence of a constant $c > 0$, which is independent of all essential parameters, such that $f \leq c \, g$. The symbol $f \approx g$ means that $f \lesssim g$ and $g \lesssim f$.}
\begin{equation}
\sup_{t>0}K(t,f;X_{0},X_{1})\approx\left\Vert f\right\Vert _{X_{0}},\text{
\ \ \ \ \ \ }\sup_{t>0}\frac{K(t,f;X_{0},X_{1})}{t}\approx\left\Vert
f\right\Vert _{X_{1}}. \label{m4}%
\end{equation}
Furthermore, if the pair is
\textquotedblleft exact\textquotedblright \,  then one can replace the equivalences signs in \eqref{m4} by equalities\footnote{From this point of view, the constants that appear in (\ref{BBM})
and (\ref{MS}) are connected to the computation of the underlying
$K$-functionals.} (cf. \cite{Milman}, \cite{KaradzhovMilmanXiao}).

From the interpolation point of view, (\ref{BBM}) and (\ref{MS}) are not
separate limiting results, but part of the recovery of end point norms of a
particular interpolation scale. Moreover, we note that the limiting formulae
(\ref{m1}) are somewhat stronger, in the sense that we are free to choose the
index $q,$ which indeed, is not dependent on the original pair of spaces, as
it is the case in (\ref{BBM}) or (\ref{MS}), where $q=p.$

In fact, the normalization constants produced by interpolation are
\textquotedblleft universal\textquotedblright, in the sense that they are
independent of the originating pair of spaces and have an explicit
formula\footnote{There is a method to define and compute the normalizations in
order insure the continuity of the interpolation scales. In particular, the normalizations
for other methods of interpolation are also known. For example, for the
complex method of Calder\'{o}n, $c_{s}=1;$ for the $J-$method, $c_{s,q}%
=\left(  (1-s)sq^{\prime}\right)  ^{-1/q^{\prime}},$ where $\frac{1}{q}%
+\frac{1}{q^{\prime}}=1.$ In the case of quasi-Banach spaces some adjustments
are necessary (cf. \cite[pag. 35]{JawerthMilman} for more details).}%
\begin{equation}
c_{s,q}=\left\{
\begin{array}
[c]{cc}%
\left(  s(1-s)q\right)  ^{\frac{1}{q}}, & \quad s\in(0,1),\quad q\in\lbrack1,\infty),\\
1, & q=\infty.
\end{array}
\right.  \label{m2}%
\end{equation}
The correct formulae for the normalizations originate from considerations
arising in the extrapolation theory of \cite{JawerthMilman}, where one needs
to work with continuous scales, and must pay close attention to the constants
that appear in the estimates. It turns out that normalized interpolation norms
have certain crucial monotonicity properties that will play a role in what
follows. In particular, it was shown in \cite[(3.2), page 19]{JawerthMilman}
that, if $q\leq r,$ then%
\[
c_{s,r}\left\Vert f\right\Vert _{(X_{0},X_{1})_{s,r}}\leq c_{s,q} \left\Vert
f\right\Vert _{(X_{0},X_{1})_{s,q}}.
\]
Furthermore, the usual reiteration formulae of interpolation theory, as
exemplified by Holmstedt's reiteration formula (cf. \cite[Theorem 2.1, page
307]{BennettSharpley}, \cite[(3.15), page 33]{JawerthMilman}, \cite[Lemmas
2.1--2.3, pages 61--63]{KaradzhovMilman}, \cite{KaradzhovMilmanXiao}, etc.) can be
sharpened, when the norms have been normalized according to \eqref{m2}.

Combining (\ref{m1}) and the fact that%
\begin{equation}
(L^{p}(\mathbb{R}^{n}),\dot{W}_{p}^{1}(\mathbb{R}^{n}))_{s,p}=\dot{W}%
^{s,p}(\mathbb{R}^{n}), \label{m3}%
\end{equation}
with hidden equivalence constants independent of $s,$ we can derive
versions of (\ref{BBM}) and (\ref{MS}). The proof of (\ref{m3}) hinges upon
the known computation of the underlying $K$-functional in terms of a
$p-$modulus of continuity (cf. \cite[Theorem 4.12, page 339]{BennettSharpley}
and the references therein),%

\begin{align*}
K(u,f;L^{p}(\mathbb{R}^{n}),\dot{W}_{p}^{1}(\mathbb{R}^{n}))  &
:=\inf_{f=f_{0}+f_{1}}\{\left\Vert f_{0}\right\Vert _{L^{p}(\mathbb{R}^{n}%
)}+u\left\Vert \nabla f_{1}\right\Vert _{L^{p}(\mathbb{R}^{n})}\}\\
&  \approx\sup_{\left\vert h\right\vert \leq u}\| \Delta_{h}%
^{1}f\|_{L^{p}(\mathbb{R}^{n})},
\end{align*}
where%
\[
\Delta_{h}^{1}f(x)=\Delta_{h}f(x)=f(x+h)-f(x),\qquad x, h\in{\mathbb{R}}^{n}.
\]
The higher order differences $\Delta_{h}^{k}, \, k \in \mathbb{N},$ are defined inductively,
\[
\Delta_{h}^{k+1}=\Delta_{h}\Delta_{h}^{k},
\]
and we can write
\begin{equation}
\Delta_{h}^{k}f(x)=\sum_{j=0}^{k}(-1)^{j}{\binom{k}{j}}f(x+(k-j)h).
\label{ClassicDiff}%
\end{equation}
The corresponding $K$-functionals for higher order Sobolev spaces $\dot{W}^k_p(\R^n)$  are given by (cf. \cite[Theorem
4.12, pag. 339]{BennettSharpley}, \cite{KaradzhovMilmanXiao}, and the
references therein)
\begin{align}
K(u^k,f;L^{p}(\mathbb{R}^{n}),\dot{W}^{k}_{p}(\mathbb{R}^{n}))  &  \approx
\sup_{\left\vert h\right\vert \leq u} \| \Delta_{h}%
^{k}f\|_{L^{p}(\mathbb{R}^{n})},\label{mod1}\\
K(u^k,f;L^{p}(\mathbb{R}^{n}),W^{k}_{p}(\mathbb{R}^{n}))  &  \approx
\min\{1,u^k\}\left\Vert f\right\Vert _{L^{p}(\mathbb{R}^{n})}+\sup_{\left\vert
h\right\vert \leq u} \| \Delta_{h}^{k}f\| _{L^{p}%
(\mathbb{R}^{n})}. \label{mod2}%
\end{align}
Combining (\ref{mod1}) and (\ref{mod2}) with (\ref{m1}), we can prove
Bourgain--Brezis--Mironescu--Maz'ya--Shaposhnikova higher order limit results in
a unified fashion\footnote{The proof of property (\ref{m4}) related to $(L^p(\R^n), \dot{W}^k_p(\R^n))$ can be seen, e.g., in
\cite[Proposition 3, pag 139]{Stein}.} (cf. \cite{KaradzhovMilmanXiao}).

Very recently, Brazke, Schikorra and Yung \cite{Braz}, using heavy Harmonic
Analysis machinery, obtained a new approach to (\ref{BBM}) and suggested a
possible extension to fractional Sobolev spaces, through the use of
Triebel--Lizorkin spaces. The paper \cite{Braz} contains a number of
interesting new inequalities and poses a number of intriguing open problems.
It was their paper that provided the initial impetus for the present work.
However, in order to solve some of the main questions asked in \cite{Braz} we
will rely on interpolation and extrapolation methods. Interestingly, as we
hope to show, interpolation theory will provide us with the crucial spaces,
and the appropriate scalings, to resolve the outstanding issues.

The first problem proposed in \cite[page 5]{Braz} that we shall discuss here
is the extension of (\ref{BBM}) to the fractional case. The interpolation
method provides the strategy, although to implement it requires the
introduction of new spaces as we now explain.

Let $t>0,$ and consider the fractional Sobolev spaces via the fractional Laplacian\footnote{$\dot{H}%
^{t,p}(\mathbb{R}^{n})$ is also known as the space of Riesz potentials, while
$H^{t,p}(\mathbb{R}^{n})$ is the space of Bessel potentials (cf.
\cite{Stein}).}%
\[
\dot{H}^{t,p}(\mathbb{R}^{n})=\{f:\left\Vert f\right\Vert _{\dot{H}%
^{t,p}(\mathbb{R}^{n})}=\| (-\Delta)^{\frac{t}{2}}f\|_{L^{p}%
(\mathbb{R}^{n})}<\infty\},
\]%
\[
H^{t,p}(\mathbb{R}^{n})=\{f:\left\Vert f\right\Vert _{H^{t,p}(\mathbb{R}^{n}%
)}=\left\Vert f\right\Vert _{L^{p}(\mathbb{R}^{n})}+\left\Vert f\right\Vert
_{\dot{H}^{t,p}(\mathbb{R}^{n})}<\infty\},
\]
and the corresponding interpolation pairs $(L^{p}({\mathbb{R}}^{n}%
),H^{t,p}({\mathbb{R}}^{n})),$ $(L^{p}({\mathbb{R}}^{n}),\dot{H}%
^{t,p}({\mathbb{R}}^{n})).$ As a consequence of $L^p$-boundedness of the Riesz transforms,  
\begin{equation}\label{R}
\|(-\Delta)^{\frac{1}{2}} f\|_{L^p(\R^n)} \approx \|\nabla f\|_{L^p(\R^n)}, \qquad p \in (1, \infty),
\end{equation}
and thus $\dot{H}^{1, p} (\R^n) = \dot{W}^1_p(\R^n)$ and $H^{1, p} (\R^n) = W^1_p(\R^n)$.

In view of the previous discussion the program we shall follow should be
clear. Indeed, a direct application of (\ref{m1}) yields the following
extension of the Bourgain--Brezis--Mironescu--Maz'ya--Shaposhnikova limits: Let
$t>0,s\in(0,1),p>1,$ then for $f\in H^{t,p}({\mathbb{R}}^{n})$,
\begin{equation}
\lim_{s\rightarrow1^{-}}(1-s)^{\frac{1}{p}}\Vert f\Vert_{(L^{p}({\mathbb{R}}%
^{n}),\dot{H}^{t,p}({\mathbb{R}}^{n}))_{s,p}}\approx\|(-\Delta
)^{\frac{t}{2}}f\| _{L^{p}({\mathbb{R}}^{n})},\lim_{s\rightarrow0^{+}}%
s^{\frac{1}{p}}\Vert f\Vert_{(L^{p}({\mathbb{R}}^{n}),\dot{H}^{t,p}%
({\mathbb{R}}^{n}))_{s,p}}\approx\Vert f\Vert_{L^{p}({\mathbb{R}}^{n})},
\label{agreg1}%
\end{equation}
with a corresponding result for the inhomogeneous pair $(L^{p}({\mathbb{R}}^{n}%
),H^{t,p}({\mathbb{R}}^{n}))$:%
\begin{equation}
\lim_{s\rightarrow1^{-}}(1-s)^{\frac{1}{p}}\Vert f\Vert_{(L^{p}({\mathbb{R}}%
^{n}),H^{t,p}({\mathbb{R}}^{n}))_{s,p}}\approx\left\Vert f\right\Vert
_{H^{t,p}(\mathbb{R}^{n})}, \lim_{s\rightarrow 0^{+}}s^{\frac{1}{p}}\Vert
f\Vert_{(L^{p}({\mathbb{R}}^{n}),H^{t,p}({\mathbb{R}}^{n}))_{s,p}}\approx\Vert
f\Vert_{L^{p}({\mathbb{R}}^{n})}. \label{agreg2}%
\end{equation}

As we noted before, the case $t=1$ in (\ref{agreg1}) (resp. (\ref{agreg2}))
corresponds to (\ref{BBM}) and (\ref{MS}), while integer values can be seen to
correspond to the higher order limit theorems of \cite{KaradzhovMilmanXiao}.
However, to obtain meaningful statements for arbitrary $t>0,$ we are required
to find explicit descriptions of the indicated interpolation spaces, and
control the corresponding interpolation norms within constants independent of
the parameter $s.$ The difficulty here is due to the fact that the nonlocal operator
$(-\Delta)^{\frac{t}{2}}$ is involved. We will overcome these obstructions using a
tool from Approximation Theory: Fractional differences. Fractional differences
give an appropriate extension of (\ref{ClassicDiff}) and lead to a natural
extension of (\ref{mod1}) and (\ref{mod2}).

We now introduce fractional differences and a class of seminorms that extends
the classical Gagliardo seminorms \eqref{gag1} to the fractional case.

\subsection{Interpolation and fractional differences: The Butzer seminorms}
Butzer and his collaborators (cf. \cite{Butzer, ButzerW} and the references therein)
apparently introduced the idea of modifying (\ref{ClassicDiff}) in order to
define fractional differences of any order $t>0$ so that a related fractional differentiation approach can be developed. For $t>0,$ we let
$\Delta_{h}^{t}$ denote the Butzer operator\footnote{The concept of fractional
differences was already used by Butzer et al. \cite{Butzer} to introduce
fractional order moduli of smoothness, which has recently become a powerful
tool in approximation theory, e.g., in the study of sharp inequalities between
moduli of smoothness (Jackson--Marchaud--Ulyanov inequalities), cf.
\cite{Tikhonov}, \cite{Trebels}, \cite{Kolomoitsev}, \cite{Kolomoitsev21} and
the references therein.} given by%
\begin{equation}
\Delta_{h}^{t}f(x)=\sum_{j=0}^{\infty}(-1)^{j}{\binom{t}{j}}f(x+(t-j)h),
\label{FractDiff}%
\end{equation}
where by definition
\[
{\binom{t}{j}}=\frac{t(t-1)\ldots(t-j+1)}{j!},\qquad {\binom{t}{0}}=1.
\]
Obviously if $t=k\in{{\mathbb{N}}}$, then ${\binom{t}{j}=0}$ for all $j\geq
k+1,$ and (\ref{FractDiff}) coincides with the classical differences
\eqref{ClassicDiff}. Remarkably, the $K$-functionals, we are seeking to
compute in relation with \eqref{agreg1}-\eqref{agreg2}, can be formally obtained\footnote{For the details of the computation
we refer to \cite{Kolomoitsev}, with \cite{Butzer} and \cite{Wilmes} as a
forerunners.} by replacing $\Delta_{h}^{k}$ by $\Delta_{h}^{t}$ on the right
hand side of (\ref{mod1}) and (\ref{mod2}), to obtain
\begin{equation*}
K(u^{t},f;L^{p}(\mathbb{R}^{n}),\dot{H}^{t,p}(\mathbb{R}^{n}))\approx
\sup_{\left\vert h\right\vert \leq u}\left\Vert \Delta_{h}^{t}f\right\Vert
_{L^{p}(\mathbb{R}^{n})}, 
\end{equation*}%
\begin{equation*}
K(u^{t},f;L^{p}(\mathbb{R}^{n}),H^{t,p}(\mathbb{R}^{n}))\approx\min
\{1,u^{t}\}\left\Vert f\right\Vert _{L^{p}(\mathbb{R}^{n})}+\sup_{\left\vert
h\right\vert \leq u}\left\Vert \Delta_{h}^{t}f\right\Vert _{L^{p}%
(\mathbb{R}^{n})}. 
\end{equation*}
Armed with the formulae for the $K$-functionals we can give a complete
description of the interpolation spaces $(L^{p},\dot{H}^{t,p}({\mathbb{R}}%
^{n}))_{s,q}$ (resp. $(L^{p},H^{t,p}({\mathbb{R}}^{n}))_{s,q}$) (cf. Lemma
\ref{LemmaIntFractGS} below) in terms of the following (semi)norms.

\begin{definition}
Let\footnote{In this paper we shall be interested in the case
$1<p<\infty.$} $1\leq p\leq\infty,0<s<t,$ then we define the Butzer seminorms,
as follows\footnote{Note also that \eqref{DefFractGS} enables us to introduce
Sobolev spaces of any order $s>0$ (not necessarily $s<1$ like in
\eqref{Gagliardo}), via higher order differences. For instance, one can
introduce for $s\in(0,2)$,
\[
\Vert f\Vert_{\dot{W}^{s,p}({\mathbb{R}}^{n}),2}=\bigg(\int_{{\mathbb{R}}^{n}%
}\int_{{\mathbb{R}}^{n}}\frac{|f(x+2h)-2f(x+h)+f(x)|^{p}}{|h|^{sp+N}%
}\,dx\,dh\bigg)^{\frac{1}{p}}.
\]
}
\begin{equation}
\Vert f\Vert_{\dot{W}^{s,p}({\mathbb{R}}^{n}),t}:=\left\{
\begin{array}
[c]{cc}%
\big(\int_{{\mathbb{R}}^{n}}\int_{{\mathbb{R}}^{n}}\frac{|\Delta_{h}%
^{t}f(x)|^{p}}{|h|^{sp+n}}\,dx\,dh \big)^{\frac{1}{p}}, &\quad  p<\infty\\
\sup_{(x,h)}  \frac{|\Delta_{h}^{t}f(x)|}{|h|^{s}},  & \quad p=\infty.
\end{array}
\right.   \label{DefFractGS}%
\end{equation}
We let\footnote{The appearence of the prefactor $(s(t-s)p)^{-\frac{1}{p}}$ in front of $\|f\|_{L^p(\R^n)}$ is dictated by standard normalizations in interpolation theory. This will become clear later, cf. Lemma \ref{LemmaIntFractGS}.} 
\begin{equation}\label{DefFractGSIn}%
\|f\|_{W^{s, p}(\R^n), t} := (s(t-s)p)^{-\frac{1}{p}} \|f\|_{L^p(\R^n)} + \|f\|_{\dot{W}^{s, p}(\R^n), t}.
\end{equation}
\end{definition}

Observe that when $t=1$ we recover the classical Gagliardo seminorms (cf. \eqref{gag1})
\begin{equation}
\Vert f\Vert_{\dot{W}^{s,p}({\mathbb{R}}^{n}),1}=\Vert f\Vert_{\dot{W}%
^{s,p}({\mathbb{R}}^{n})},\qquad s\in(0,1). \label{FractModuInt}%
\end{equation}
Furthermore, if $s<t<1$ then $\Vert\cdot\Vert_{\dot{W}^{s,p}({\mathbb{R}}%
^{n}),t}$ defines an equivalent semi-norm on $\dot{W}^{s,p}({\mathbb{R}}%
^{n})$, i.e.,
\begin{equation}
\Vert f\Vert_{\dot{W}^{s,p}({\mathbb{R}}^{n}),t}\approx\Vert f\Vert_{\dot
{W}^{s,p}({\mathbb{R}}^{n})}. \label{EquivalenceGSnorms}%
\end{equation}
Note, however, that the constants of equivalence blow up, as $s$ approaches
$t$ (cf. Lemma \ref{LemmaFractInt} below):
\begin{equation}
\frac{1}{(t-s)^{\frac{1}{\max\{p,2\}}}}\Vert f\Vert_{\dot{W}^{s,p}({\mathbb{R}}^{n}%
)}\lesssim\Vert f\Vert_{\dot{W}^{s,p}({\mathbb{R}}^{n}),t}\lesssim\frac
{1}{(t-s)^{\frac{1}{\min\{p,2\}}}}\Vert f\Vert_{\dot{W}^{s,p}({\mathbb{R}}^{n})}.
\label{later}%
\end{equation}

By direct computation we characterize the interpolation norms (cf. Lemma \ref{LemmaIntFractGS}
below): Suppose that $1<p<\infty,0<s<1$ and $t>0$ then%
\begin{equation*}
\left\Vert f\right\Vert _{(L^{p}({\mathbb{R}}^{n}),\dot{H}^{t,p}({\mathbb{R}%
}^{n}))_{s,p}}\approx\left\Vert f\right\Vert _{\dot{W}^{st,p}({\mathbb{R}}%
^{n}),t}, 
\end{equation*}
with hidden constants of equivalence independent of $s$ (but depending on
$n,p$ and $t$).

In particular, to derive the fractional
Bourgain--Brezis--Mironescu--Maz'ya--Shaposhnikova limit formulae, we simply
replace the interpolation seminorms $\|\cdot\|_{(L^{p},\dot
{H}^{t,p}({\mathbb{R}}^{n}))_{s,p}}$ by $\|\cdot\|_{\dot
{W}^{st,p}({\mathbb{R}}^{n}),t}$ in (\ref{agreg1}) (resp., replace $\| \cdot\|_{(L^{p},H^{t,p}({\mathbb{R}}^{n}))_{s,p}}$ with $\|\cdot\|_{\dot{W}^{st,p}({\mathbb{R}}^{n}),t}+ (s(1-s))^{-\frac{1}{p}} \|\cdot\|
_{L^{p}({\mathbb{R}}^{n})}$ in (\ref{agreg2})$).$ For example, the general
fractional homogeneous version reads now
%\[
%\lim_{s\rightarrow1^{-}}(1-s)^{\frac{1}{p}}\left\Vert f\right\Vert _{\dot
%{W}^{st,p}({\mathbb{R}}^{n}),t}\approx\| (-\Delta)^{\frac{t}{2}}f\|_{L^{p}({\mathbb{R}}^{n})},\qquad \lim_{s\rightarrow0^{+}}s^{\frac{1}{p}}\left\Vert
%f\right\Vert _{\dot{W}^{st,p}({\mathbb{R}}^{n}),t}\approx\Vert f\Vert
%_{L^{p}({\mathbb{R}}^{n})}.
%\]
%Note that, in view of (\ref{neces}), we can rewrite this result in the
%following equivalent way%
\begin{equation}
\lim_{s\rightarrow t^{-}}(t-s)^{\frac{1}{p}}\Vert f\Vert_{\dot{W}^{s,p}%
({\mathbb{R}}^{n}),t}\approx \| (-\Delta)^{\frac{t}{2}} f\|_{L^p
(\mathbb{R}^{n})}, \qquad \lim_{s\rightarrow0^{+}}s^{\frac{1}{p}}\Vert f\Vert_{\dot
{W}^{s,p}({\mathbb{R}}^{n}),t}\approx\Vert f\Vert_{L^{p}({\mathbb{R}}^{n})}.
\label{bbmf}%
\end{equation}

This answers completely the question raised in \cite[page 5]{Braz} asking for a
possible fractional extension of (\ref{BBM}) (see also \eqref{R}). Note that
(\ref{bbmf}) and (\ref{later}) show that we cannot use the classical Gagliardo
seminorms for this purpose.

Further evidence of the important r\^{o}le of the Butzer seminorms in our
project is provided by the following result, which answers another question
raised in \cite{Braz}.

\subsection{The fractional Bourgain--Brezis--Mironescu convergence theorem via
Butzer seminorms}

We consider the extension to fractional Sobolev spaces of the
Bourgain--Brezis--Mironescu convergence theorem \cite{BBM},

\begin{theorem}
\label{convergencebbm}Let $p\in(1,\infty)$, assume that $f_{k}\in
\mathcal{S}({\mathbb{R}}^{n})$, $k\in\mathbb{N}$, and%
\[
f_{k}\rightharpoonup f\quad\text{weakly in}\quad L^{p}({\mathbb{R}}^{n}%
)\quad\text{as}\quad k\rightarrow\infty.
\]
Let $\{s_{k}\}_{k\in{{\mathbb{N}}}}\subset(0,1)$ be such that $s_{k}\uparrow1$
and, moreover, assume that
\[
\Lambda:=\sup_{k\in{{\mathbb{N}}}}\Big(\Vert f_{k}\Vert_{L^{p}({\mathbb{R}%
}^{n})}+(1-s_{k})^{\frac{1}{p}}\Vert f_{k}\Vert_{\dot{W}^{s_{k},p}%
({\mathbb{R}}^{n})}\Big)<\infty.
\]
Then $f\in W_{p}^{1}({\mathbb{R}}^{n})$ and there exists $C=C(p,n)>0$ such
that
\[
\Vert f\Vert_{L^{p}({\mathbb{R}}^{n})}+\Vert\nabla f\Vert_{L^{p}({\mathbb{R}%
}^{n})}\leq C\Lambda.
\]
In fact, $f_{k}\rightarrow f$ strongly in $L_{loc}^{p}({\mathbb{R}}^{n}).$
\end{theorem}

In \cite{Braz} the authors propose a plan of attack for a possible fractional
extension of Theorem \ref{convergencebbm}. The idea is based on the use of the 
Triebel-Lizorkin spaces $\dot{F}^s_{p, q}(\R^N)$ (we refer to Section \ref{secfunctionspaces} for the  definitions). Indeed, based on this idea, a new proof of Theorem
\ref{convergencebbm} was presented in \cite{Braz}, using the following sharp Sobolev-type
inequality:

\begin{theorem}
\label{Theorem343} Let $s\in(0,1),p\in(1,\infty)$ and $\Lambda>1$. Assume
$1-s\leq\frac{1}{2\Lambda}$ and let $1-\bar{r}=\Lambda(1-s)$. Then there
exists $C$, which depends only on $n,p$ and $\Lambda$, such that
\[
\Vert f\Vert_{\dot{F}_{p,2}^{r}({\mathbb{R}}^{n})}\leq C\Big(\Vert
f\Vert_{L^{p}({\mathbb{R}}^{n})}+(1-s)^{\frac{1}{p}}\Vert f\Vert_{\dot
{W}^{s,p}({\mathbb{R}}^{n})}\Big),\qquad r\in\lbrack0,\bar{r}].
\]

\end{theorem}

With this result at their disposal the proof of Theorem \ref{convergencebbm}
given in \cite{Braz} now proceeds by combining Theorem \ref{Theorem343}
(taking limits as $r\rightarrow1^{-})$, with \eqref{R} and the Littewood--Paley estimate
\begin{equation*}
\Vert f\Vert_{\dot{F}_{p,2}^{t}({\mathbb{R}}^{n})}\approx\Vert(-\Delta
)^{\frac{t}{2}}f\Vert_{L^{p}({\mathbb{R}}^{n})}, \qquad1<p<\infty, 
\end{equation*}
and then applying the Rellich--Kondrachov theorem.

In view of this, the following conjecture was formulated in \cite{Braz}.

\begin{op}
[{\cite[Question 1.11]{Braz}}]\label{Question1.11} Let $p\in(1,\infty)$,
assume that
\[
f_{k}\rightharpoonup f\quad\text{weakly in}\quad L^{p}({\mathbb{R}}^{n}%
)\quad\text{as}\quad k\rightarrow\infty.
\]
Let $t\in(0,1)$ and $(s_{k})_{k\in{{\mathbb{N}}}}\subset(0,t)$ such that
$s_{k}\uparrow t$ and assume that
\[
\Lambda:=\sup_{k\in{{\mathbb{N}}}}\Big(\Vert f_{k}\Vert_{L^{p}({\mathbb{R}%
}^{n})}+(t-s_{k})^{\frac{1}{p}}\Vert f_{k}\Vert_{\dot{W}^{s_{k},p}%
({\mathbb{R}}^{n})}\Big)<\infty.
\]
Then $f\in H^{t,p}({\mathbb{R}}^{n})$ and there exists $C=C(p,n,t)>0$ such
that
\[
\lim_{\bar{t}\uparrow t}\limsup_{k\rightarrow\infty}\Vert(-\Delta)^{\frac
{\bar{t}}{2}}f_{k}\Vert_{L^{p}({\mathbb{R}}^{n})}\leq C\Lambda.
\]
\end{op}

Note that in the case $t=1,$ the conjectured result is exactly Theorem
\ref{convergencebbm}. However, as we have indicated above, the standard Gagliardo
seminorms are \textquotedblleft too far" from the \textquotedblleft
exact\textquotedblright\ interpolation seminorms $\|\cdot\|_{\dot{W}^{s_{k},p}({\mathbb{R}%
}^{n}), t} \,$ (cf. \eqref{DefFractGS}) for $0<t<1$. We use this insight to give a counterexample to Open Problem \ref{Question1.11} (cf.
Section \ref{SectionMR}, Proposition \ref{PropQuestion1.11} for the precise
statement, and Section \ref{SectionProofs} for a proof).

The problem thus remains: What spaces should we use to formulate and prove a
fractional extension of Theorem \ref{convergencebbm}? Once again
interpolation, via the Butzer seminorms, comes to our rescue and we are
essentially able to prove that the conjectured result in Open Problem \ref{Question1.11}  is true if in its
statement we replace  \textquotedblleft Gagliardo seminorms\textquotedblright%
\ by \textquotedblleft Butzer seminorms\textquotedblright\ (cf. Section
\ref{SectionMR}, Theorem \ref{ThemQuestion1.11} for the precise statement, and
Section \ref{SectionProofs} for a complete proof).

\subsection{Sharp Sobolev inequalities}

Another problem proposed in \cite{Braz} deals with Sobolev type inequalities.
Let $0<s<t<1$, an elementary version of Sobolev's inequality asserts that
\begin{equation}
\Vert f\Vert_{\dot{W}^{s,2}({\mathbb{R}}^{n})}\leq\gamma_{s,t}\Big(\Vert
f\Vert_{L^{2}({\mathbb{R}}^{n})}+\Vert f\Vert_{\dot{W}^{t,2}({\mathbb{R}}%
^{n})}\Big). \label{SobolevHilbert}%
\end{equation}
Here $\gamma_{s,t}$ is a positive constant which depends, in particular, on
the smoothness parameters $s$ and $t$. Furthermore, at least formally,
\eqref{SobolevHilbert} is still valid in the limiting values $s=0$ and $t=1$.
According to \eqref{Constancy}, \eqref{BBM} and \eqref{MS}, the validity of
the previous assertion should be reflected in the behavior of the equivalence
constant $\gamma_{s,t}$ that appears in \eqref{SobolevHilbert}, in terms of
the blows up $\min\{t,1-t\}^{\frac{1}{2}}$ and $\min\{s,1-s\}^{\frac{1}{2}}$.
Indeed, the following result was obtained in \cite[Corollary 1.3]{Braz}.

\begin{theorem}
\label{ThmIntro1} Let $0<s<t<1$ and $f\in\mathcal{S}({\mathbb{R}}^{n})$. Then
there exists $C=C(n)>0$, such that
\[
\min\{s,1-s\}^{\frac{1}{2}}\Vert f\Vert_{\dot{W}^{s,2}({\mathbb{R}}^{n})}\leq
C\Big(\Vert f\Vert_{L^{2}({\mathbb{R}}^{n})}+\min\{t,1-t\}^{\frac{1}{2}}\Vert
f\Vert_{\dot{W}^{t,2}({\mathbb{R}}^{n})}\Big).
\]

\end{theorem}

It is well known that the $L^{p}$-counterpart of \eqref{SobolevHilbert} is
also true, i.e.,
\begin{equation*}
\Vert f\Vert_{\dot{W}^{s,p}({\mathbb{R}}^{n})}\leq C_{s,t}\Big(\Vert
f\Vert_{L^{p}({\mathbb{R}}^{n})}+\Vert f\Vert_{\dot{W}^{t,p}({\mathbb{R}}%
^{n})}\Big)
\end{equation*}
provided that $0<s<t<1$ and $1<p<\infty$. Accordingly, a similar result, in
the spirit of Theorem \ref{ThmIntro1}, was expected. However, the set of
harmonic analysis techniques developed in \cite{Braz} do not work outside the
Hilbert setting given by $p=2$. This leads to the following question, explicitly raised in
\cite{Braz}:

\begin{op}
\label{OpenProblemIntro1} Let $p\in(1,\infty),0<\theta<s<t<1$ and
$f\in\mathcal{S}({\mathbb{R}}^{n})$. Does there exist $C=C(n,p,\theta)>0$,
such that%
\begin{equation}
\min\{s,1-s\}^{\frac{1}{p}}\Vert f\Vert_{\dot{W}^{s,p}({\mathbb{R}}^{n})}\leq
C\Big(\Vert f\Vert_{L^{p}({\mathbb{R}}^{n})}+\min\{t,1-t\}^{\frac{1}{p}}\Vert
f\Vert_{\dot{W}^{t,p}({\mathbb{R}}^{n})}\Big)? \label{OpenProblemIntro1Eq}%
\end{equation}

\end{op}

We provide a positive answer to this problem. Indeed, in Section
\ref{SectionMR}, Theorem \ref{ConjectureBSY}, we state a somewhat stronger
result which is then proved in Section \ref{SectionProofs}.

We now turn to explain the organization of the paper. The brief Section
\ref{secfunctionspaces} contains the basic background and notation concerning
the function spaces we shall consider; Section \ref{SectionMR} contains the
statements of the main results, which we then prove in Section
\ref{SectionProofs}. The interpolating Section \ref{SectionMethod} contains
detailed information with complete proofs concerning the interpolation and
extrapolation tools that constitute the basic methods we use in this paper.
Finally, in the Appendix, we collect further results and applications. In particular, we prove an extension of the classical fractional Sobolev inequalities of Bourgain--Brezis--Mironescu \cite{BBM1} and Maz'ya--Shaposhnikova \cite{Mazya}, that were one of the original motivations for the limiting theorems \eqref{BBM} and \eqref{MS}.

\vspace{4mm}
\textbf{Acknowledgments.} We are grateful to Petru Mironescu for his constant interest and encouragement during the preparation of this paper and to Sergey Tikhonov for precious information concerning fractional differences and moduli of smoothness.

 The first named author has been partially supported by the LABEX MILYON (ANR-10-LABX-0070) of Universit\'e de Lyon, within the program ``Investissement d'Avenir" (ANR-11-IDEX-0007) operated by the
French National Research Agency (ANR), and by MTM2017-84058-P (AEI/FEDER, UE) .

\section{Background on function spaces\label{secfunctionspaces}}

Let $s\in(0,1)$ and $p\in(1,\infty)$, the \emph{(fractional) Sobolev space} $\dot
{W}^{s,p}({\mathbb{R}}^{n})$ is equipped with the \emph{Gagliardo seminorm}
\begin{equation}
\Vert f\Vert_{\dot{W}^{s,p}({\mathbb{R}}^{n})}:=\bigg(\int_{{\mathbb{R}}^{n}%
}\int_{{\mathbb{R}}^{n}}\frac{|f(x)-f(y)|^{p}}{|x-y|^{sp+n}}%
\,dx\,dy\bigg)^{1/p}. \label{Gagliardo}%
\end{equation}
%Let $W^{s, p}(\R^n) = L^p(\R^n) \cap \dot{W}^{s, p}(\R^n)$ and $\|f\|_{W^{s, p}(\R^n)} = \|f\|_{L^p(\R^n)} + \|f\|_{\dot{W}^{s,p}({\mathbb{R}}^{n})}$.

The \emph{Riesz potential space} $\dot{H}^{s,p}({\mathbb{R}}^{n})$ is endowed with
\[
\Vert f\Vert_{\dot{H}^{s,p}({\mathbb{R}}^{n})}:=\Vert(-\Delta)^{\frac{s}{2}%
}f\Vert_{L^{p}({\mathbb{R}}^{n})}.
\]
The spaces $\dot{H}^{s,p}({\mathbb{R}}^{n})$ make sense for any $s\in
{\mathbb{R}}$. In particular, $\dot{H}^{k,p}({\mathbb{R}}^{n}),\,k\in
{{\mathbb{N}}},$ coincides with the \emph{classical Sobolev space} $\dot{W}_{p}%
^{k}({\mathbb{R}}^{n})$ and
\begin{equation}
\Vert f\Vert_{\dot{H}^{k,p}({\mathbb{R}}^{n})}\approx\Vert f\Vert_{\dot{W}%
_{p}^{k}({\mathbb{R}}^{n})}:=\|\nabla^{k}f\|_{L^{p}({\mathbb{R}}^{n}%
)},\qquad p\in(1,\infty). \label{SobRiesz}%
\end{equation}
We let $H^{s,p}({\mathbb{R}}^{n})=L^{p}({\mathbb{R}}^{n})\cap\dot{H}%
^{s,p}({\mathbb{R}}^{n})$ (the \emph{space of Bessel potentials}) and $\|f\|_{H^{s,p}({\mathbb{R}}^{n})} = \|f\|_{L^p(\R^n)} + \Vert(-\Delta)^{\frac{s}{2}%
}f\Vert_{L^{p}({\mathbb{R}}^{n})}$.

The \emph{(homogeneous) Triebel--Lizorkin space} $\dot{F}_{p,q}^{s}({\mathbb{R}}%
^{n}),\,s\in{\mathbb{R}},\,p\in(0,\infty),\,q\in(0,\infty]$, is formed by all
$f\in\dot{\mathcal{S}}^{\prime}({\mathbb{R}}^{n})$ such that
\[
\Vert f\Vert_{\dot{F}_{p,q}^{s}({\mathbb{R}}^{n})}:=\bigg\|\bigg(\sum
_{j=-\infty}^{\infty}2^{jsq}|\Delta_{j}f|^{q}\bigg)^{1/q}\bigg\|_{L^{p}%
({\mathbb{R}}^{n})}<\infty,
\]
with the obvious modification if $q=\infty$. Here, $\{\Delta_{j}%
f:j\in{{\mathbb{Z}}}\}$ is the standard dyadic Littlewood--Paley decomposition
of $f$. The inhomogeneous counterparts, $F_{p,q}^{s}({\mathbb{R}}^{n}),$ can
be introduced similarly.

Interchanging the roles of $L^{p}({\mathbb{R}}^{n})$ and $\ell_{q}%
({{\mathbb{Z}}})$ in the definition of $\dot{F}_{p,q}^{s}({\mathbb{R}}^{n})$
we arrive at the Besov spaces. Specifically, the \emph{(homogeneous) Besov space}
$\dot{B}_{p,q}^{s}({\mathbb{R}}^{n}),\,s\in{\mathbb{R}},p,q\in(0,\infty]$, is
formed by all $f\in\dot{\mathcal{S}}^{\prime}({\mathbb{R}}^{n})$ such that
\[
\Vert f\Vert_{\dot{B}_{p,q}^{s}({\mathbb{R}}^{n})}:=\bigg(\sum_{j=-\infty
}^{\infty}2^{jsq}\Vert\Delta_{j}f\Vert_{L^{p}({\mathbb{R}}^{n})}%
^{q}\bigg)^{1/q}<\infty,
\]
with the obvious modifications if $q=\infty$.

We shall use the following standard notations. Given a quasi-Banach space $X$
and $\lambda>0$, we will denote by $\lambda X$ the space endowed with
\[
\Vert f\Vert_{\lambda X}=\lambda\Vert f\Vert_{X}, \qquad f \in X.
\]
Let $Y$ be a quasi-Banach space and let $\beta>0$. We will use the notation
$\lambda X\hookrightarrow\beta Y$ to indicate that there exists a positive
constant $C$, independent of $\lambda$ and $\beta$, such that
\[
\beta\Vert f\Vert_{Y}\leq C\lambda\Vert f\Vert_{X},\qquad f\in X.
\]
In particular, if $\lambda=\beta=1$ then we simply write $X\hookrightarrow Y$.
By $X=Y$ we mean that $X\hookrightarrow Y$ and $Y\hookrightarrow X$.

We collect some well-known relations between the function spaces described
above (cf. \cite{BerghLofstrom}, \cite{Stein} and \cite{Triebel83}).

\begin{lemma}
\label{TableCoincidences}

\begin{enumerate}
[\upshape(i)]

\item Let $s\in{\mathbb{R}},\,p\in(0,\infty)$ and $q\in(0,\infty]$. Then
\[
\dot{B}_{p,\min\{p,q\}}^{s}({\mathbb{R}}^{n})\hookrightarrow\dot{F}_{p,q}%
^{s}({\mathbb{R}}^{n})\hookrightarrow\dot{B}_{p,\max\{p,q\}}^{s}({\mathbb{R}%
}^{n}).
\]
In particular,
\[
\dot{F}_{p,p}^{s}({\mathbb{R}}^{n})=\dot{B}_{p,p}^{s}({\mathbb{R}}^{n}).
\]

\item Let $s\in(0,1)$ and $p\in(1,\infty)$. Then
\[
\dot{W}^{s,p}({\mathbb{R}}^{n})=\dot{B}_{p,p}^{s}({\mathbb{R}}^{n}).
\]

\item Let $s\in{\mathbb{R}}$ and $p\in(1,\infty)$. Then
\[
\dot{F}_{p,2}^{s}({\mathbb{R}}^{n})=\dot{H}^{s,p}({\mathbb{R}}^{n}).
\]

\end{enumerate}
\end{lemma}

\section{Main results}

\label{SectionMR}

In this section we provide a list of statements of the main results we have
obtained. The proofs will be given in Section \ref{SectionProofs}.

We start outlining some details of a counterexample that provides a negative
answer to Open Problem \ref{Question1.11}. To avoid technicalities, we switch
temporarily from ${\mathbb{R}}^{n}$ to the unit circle ${\mathbb{T}}$\footnote{Spaces of periodic functions are defined similarly as their analogs on $\R^n$, simply replacing $L^p(\R^n)$ by $L^p(\T)$.}.
However, our construction can be easily modified to deal with ${\mathbb{R}%
}^{n}$ and $p>\frac{2n}{n+1}$; in this connection see Remark
\ref{RemarkCounterexample} below.

\begin{proposition}
\label{PropQuestion1.11} Let $p\in(1,\infty)$ and $t\in(0,1)$. Let $f$ be
formally associated to a Fourier series as follows,
\begin{equation}
f(x)\sim\sum_{\nu=1}^{\infty}\nu^{-t-1+\frac{1}{p}}\cos(\nu x),\qquad
x\in{\mathbb{T}}.\label{FourierSeries}%
\end{equation}
Then, $f\in L^{p}({\mathbb{T}})$ but $f\not \in H^{t,p}({\mathbb{T}})$.
Furthermore, given any sequence $(s_{k})_{k\in{{\mathbb{N}}}}\subset(0,t)$
with $s_{k}\uparrow t$,
\[
\sup_{k\in{{\mathbb{N}}}}\,(t-s_{k})^{\frac{1}{p}}\Vert f\Vert_{\dot{W}%
^{s_{k},p}({\mathbb{T}})}<\infty.
\]

\end{proposition}

In fact, we are able to go far beyond, and we show that a correct
formulation of Open Problem \ref{Question1.11} is obtained by means of replacing
the classical seminorms $\|\cdot\|_{\dot{W}^{s,p}({\mathbb{R}}^{n})}$
by the Butzer seminorms $\|\cdot\|_{\dot{W}^{s,p}({\mathbb{R}}^{n}%
),t}$ (cf. \eqref{DefFractGS} and \eqref{EquivalenceGSnorms}).

\begin{theorem}
\label{ThemQuestion1.11} Let $p\in(1,\infty)$, assume that
\[
f_{k}\rightharpoonup f\quad\text{weakly in}\quad L^{p}({\mathbb{R}}^{n}%
)\quad\text{as}\quad k\rightarrow\infty.
\]
Let $t\in(0,1)$ and $(s_{k})_{k\in{{\mathbb{N}}}}\subset(0,t)$ be such that
$s_{k}\uparrow t.$ Assume that
\[
\Lambda:=\sup_{k\in{{\mathbb{N}}}}\Big(\Vert f_{k}\Vert_{L^{p}({\mathbb{R}%
}^{n})}+(t-s_{k})^{\frac{1}{p}}\Vert f_{k}\Vert_{\dot{W}^{s_{k},p}%
({\mathbb{R}}^{n}),t}\Big)<\infty.
\]
Then $f\in H^{t,p}({\mathbb{R}}^{n})$ and there exists $C=C(n, p, t)>0$ such
that
\[
\lim_{\bar{t}\uparrow t}\limsup_{k\rightarrow\infty}\Vert(-\Delta)^{\frac
{\bar{t}}{2}}f_{k}\Vert_{L^{p}({\mathbb{R}}^{n})}\leq C\Lambda.
\]

\end{theorem}

To prove Theorem \ref{ThemQuestion1.11} we need an extension of Theorem
\ref{Theorem343} formulated in terms of the Butzer seminorms
$\Vert\cdot\Vert_{\dot{W}^{s,p}({\mathbb{R}}^{n}),t}$. It turns out that our
method is flexible enough to incorporate both, the case $s\rightarrow0+$ , as
well as to deal with inhomogeneous Triebel--Lizorkin spaces.

\begin{theorem}
\label{Theorem343Fract} Let $0 < s < t \leq1, \, p \in(1, \infty)$ and
$\Lambda> 1$.

\begin{enumerate}
\item[(1)] Assume $t-s\leq\frac{t}{2\Lambda}$ and let $t-\bar{r}=\Lambda
(t-s)$. Then, there exists $C$, which depends only on $n,p,t$ and $\Lambda$,
such that
\begin{equation}
\Vert f\Vert_{F_{p,2}^{\bar{r}}({\mathbb{R}}^{n})}\leq C\Big(\Vert
f\Vert_{L^{p}({\mathbb{R}}^{n})}+(t-s)^{\frac{1}{p}}\Vert f\Vert_{\dot
{W}^{s,p}({\mathbb{R}}^{n}),t}\Big).\label{Aux61}%
\end{equation}
There is a corresponding result for homogeneous Triebel--Lizorkin spaces,
\begin{equation}
\Vert f\Vert_{\dot{F}_{p,2}^{r}({\mathbb{R}}^{n})}\leq C\Big(\Vert
f\Vert_{L^{p}({\mathbb{R}}^{n})}+(t-s)^{\frac{1}{p}}\Vert f\Vert_{\dot
{W}^{s,p}({\mathbb{R}}^{n}),t}\Big),\qquad r\in\lbrack0,\bar{r}].\label{Aux6}%
\end{equation}

\item[(2)] Assume $s<\frac{1}{2\Lambda}$ and let $\bar{r}=\frac{1}{\Lambda}s$.
Then there exists $C$, which depends only on $n,p,t$ and $\Lambda$, such that
\begin{equation}
\Vert f\Vert_{F_{p,2}^{\bar{r}}({\mathbb{R}}^{n})}\leq C\Big(\Vert
f\Vert_{L^{p}({\mathbb{R}}^{n})}+s^{\frac{1}{p}}\Vert f\Vert_{\dot{W}%
^{s,p}({\mathbb{R}}^{n}),t}\Big). \label{Aux62}%
\end{equation}
Likewise, for homogeneous Triebel--Lizorkin spaces we have
\begin{equation}
\Vert f\Vert_{\dot{F}_{p,2}^{r}({\mathbb{R}}^{n})}\leq C\Big(\Vert
f\Vert_{L^{p}({\mathbb{R}}^{n})}+s^{\frac{1}{p}}\Vert f\Vert_{\dot{W}%
^{s,p}({\mathbb{R}}^{n}),t}\Big),\qquad r\in\lbrack0,\bar{r}]. \label{Aux6new}%
\end{equation}

\end{enumerate}
\end{theorem}

We apply our method to provide an extension of \eqref{BBM}-\eqref{MS} in terms
of the fractional Laplacian and Butzer seminorms.

\begin{theorem}
\label{TheoremFBBM} Let $t\in(0,1]$ and $p\in(1,\infty)$. Assume $f\in
H^{t,p}({\mathbb{R}}^{n})$. Then
\begin{equation}
\lim_{s\rightarrow t^{-}}(t-s)^{\frac{1}{p}}\Vert f\Vert_{\dot{W}^{s,p}%
({\mathbb{R}}^{n}),t}\approx\Vert(-\Delta)^{\frac{t}{2}}f\Vert_{L^{p}%
({\mathbb{R}}^{n})} \label{FractGSAssertion}%
\end{equation}
and
\begin{equation}
\lim_{s\rightarrow0^{+}}s^{\frac{1}{p}}\Vert f\Vert_{\dot{W}^{s,p}({\mathbb{R}%
}^{n}),t}\approx\Vert f\Vert_{L^{p}({\mathbb{R}}^{n})}.
\label{FractGSAssertionMS}%
\end{equation}
The corresponding results for inhomogeneous spaces also hold true:
\begin{equation}
\lim_{s\rightarrow t^{-}}(t-s)^{\frac{1}{p}}\Vert f\Vert_{W^{s,p}%
({\mathbb{R}}^{n}),t}\approx \|f\|_{H^{t, p}(\R^n)}  \label{FractGSAssertionIn}%
\end{equation}
and
\begin{equation}
\lim_{s\rightarrow0^{+}}s^{\frac{1}{p}}\Vert f\Vert_{W^{s,p}({\mathbb{R}%
}^{n}),t}\approx\Vert f\Vert_{L^{p}({\mathbb{R}}^{n})}.
\label{FractGSAssertionMSIn}%
\end{equation}
\end{theorem}

\begin{remark}
Observe that \eqref{FractGSAssertion} and \eqref{FractGSAssertionMS} with $t=1$ give back, up to equivalence constants, the classical formulae \eqref{BBM} and \eqref{MS}.
\end{remark}

As indicated by Theorems \ref{ThemQuestion1.11} and \ref{TheoremFBBM}, the
classical seminorms $\Vert\cdot\Vert_{\dot{W}^{s,p}({\mathbb{R}}^{n})}$ are not the
optimal choices for dealing with Sobolev-type inequalities involving
$H^{t,p}({\mathbb{R}}^{n})$. This phenomenon is illustrated in the following result.

\begin{theorem}
\label{ThmSobolevBSY} Let $1<p<\infty$ and $0<r<t<1$. Then there exists
$C=C(n,p,r,t)>0,$ such that
\begin{equation}
\Vert f\Vert_{\dot{W}^{s,p}({\mathbb{R}}^{n})}\leq C((s-r)(t-s))^{\frac
{1}{\max\{p,2\}}}\bigg(\frac{1}{(s-r)^{\frac{1}{p}}}\Vert f\Vert_{\dot
{F}_{p,2}^{r}({\mathbb{R}}^{n})}+\frac{1}{(t-s)^{\frac{1}{p}}}\Vert
f\Vert_{\dot{F}_{p,2}^{t}({\mathbb{R}}^{n})}\bigg), \label{rst}%
\end{equation}
for every $s\in(r,t)$. In the limiting cases, $r=0$ and $t=1$, we have
\begin{equation}
\Vert f\Vert_{\dot{W}^{s,p}({\mathbb{R}}^{n})}\leq C(t-s)^{\frac{1}%
{\max\{p,2\}}}\bigg(\frac{1}{s^{\frac{1}{p}}}\Vert f\Vert_{L^{p}({\mathbb{R}%
}^{n})}+\frac{1}{(t-s)^{\frac{1}{p}}}\Vert(-\Delta)^{\frac{t}{2}}f\Vert
_{L^{p}({\mathbb{R}}^{n})}\bigg) \label{WLD}
\end{equation}
for $0 < t < 1$ and
\begin{equation}
\Vert f\Vert_{\dot{W}^{s,p}({\mathbb{R}}^{n})}\leq C(s-r)^{\frac{1}%
{\max\{p,2\}}}\bigg(\frac{1}{(s-r)^{\frac{1}{p}}}\Vert(-\Delta)^{\frac{r}{2}%
}f\Vert_{L^{p}({\mathbb{R}}^{n})}+\frac{1}{(1-s)^{\frac{1}{p}}}\Vert\nabla
f\Vert_{L^{p}({\mathbb{R}}^{n})}\bigg) \label{WLD2}%
\end{equation}
for $0 < r < 1$, respectively.
\end{theorem}

As an immediate consequence of \eqref{WLD} and \eqref{WLD2}, we obtain

\begin{corollary}\label{CorSharp}
Let $0<r<t<1$. Then
\begin{equation}
\sup_{s\in\lbrack r,t)}(t-s)^{\frac{1}{p}-\frac{1}{\max\{p,2\}}}\Vert
f\Vert_{\dot{W}^{s,p}({\mathbb{R}}^{n})}\leq C\Big(\Vert f\Vert_{L^{p}%
({\mathbb{R}}^{n})}+\Vert(-\Delta)^{\frac{t}{2}}f\Vert_{L^{p}({\mathbb{R}}%
^{n})}\Big) \label{Sharp}%
\end{equation}
and
\begin{equation}
\sup_{s\in(r,t]}(s-r)^{\frac{1}{p}-\frac{1}{\max\{p,2\}}}\Vert f\Vert_{\dot
{W}^{s,p}({\mathbb{R}}^{n})}\leq C\Big(\Vert(-\Delta)^{\frac{r}{2}}%
f\Vert_{L^{p}({\mathbb{R}}^{n})}+\Vert\nabla f\Vert_{L^{p}({\mathbb{R}}^{n}%
)}\Big). \label{Sharp2}%
\end{equation}

\end{corollary}

\begin{remark}
\label{RemarkOptim} Theorem \ref{ThmSobolevBSY} is a considerable improvement of Theorem 1.6 in \cite{Braz}: Let $1 < p < \infty$ and $0 \leq r < t \leq 1$. Then there exists $C = C(n, p) > 0$ such that
\begin{equation}
\Vert f\Vert_{\dot{W}^{s,p}({\mathbb{R}}^{n})}\leq C\, \bigg(\frac{1}{(s-r)^{\frac{1}{p}}}\Vert f\Vert_{\dot
{F}_{p,2}^{r}({\mathbb{R}}^{n})}+\frac{1}{(t-s)^{\frac{1}{p}}}\Vert
f\Vert_{\dot{F}_{p,2}^{t}({\mathbb{R}}^{n})}\bigg), \qquad s \in (r, t). \label{rsnew}%
\end{equation}
The estimates provided in Theorem \ref{ThmSobolevBSY} sharpens \eqref{rsnew} due to the appearance of one of the additional prefactors $((s-r)(t-s))^{\frac{1}%
{\max\{p,2\}}}, (t-s)^{\frac{1}{\max\{p, 2\}}}$ and $(s-r)^{\frac{1}{\max\{p, 2\}}}$. Note that the constant $C$ in \eqref{rsnew} does not depend on $r$ and $t$, but this is not the case in \eqref{rst}. The reason behind this additional flexibility in \eqref{rsnew} is explained by its non-optimality. 

On the other hand, \eqref{Sharp} sharpens the following inequality obtained in Corollary 1.7 in \cite{Braz}: Let $1 < p < \infty$ and $0 < r < t \leq 1$. Then there exists $C = C(n, p, r) > 0$ such that
\begin{equation}
\sup_{s\in\lbrack r,t)}(t-s)^{\frac{1}{p}}\Vert
f\Vert_{\dot{W}^{s,p}({\mathbb{R}}^{n})}\leq C\Big(\Vert f\Vert_{L^{p}%
({\mathbb{R}}^{n})}+\Vert(-\Delta)^{\frac{t}{2}}f\Vert_{L^{p}({\mathbb{R}}%
^{n})}\Big). \label{1.5}
\end{equation}
In particular, this inequality with $t=1$ gives one of the estimates in the classical Bourgain--Brezis--Mironescu formula \eqref{BBM}. However, \eqref{1.5} is not optimal if $t \in (0, 1)$. In fact, as was already observed in \cite[Remark 1.8]{Braz}, the inequality \eqref{1.5} is only useful with $p < 2$. Specifically, if $p \geq 2$ and $0 < t < 1$, it is well known that
\begin{equation}\label{Tri}
	\sup_{s\in\lbrack r,t)} \Vert
f\Vert_{\dot{W}^{s,p}({\mathbb{R}}^{n})}\leq C\Big(\Vert f\Vert_{L^{p}%
({\mathbb{R}}^{n})}+\Vert(-\Delta)^{\frac{t}{2}}f\Vert_{L^{p}({\mathbb{R}}%
^{n})}\Big).
\end{equation}
This obstruction can be overcome via the inequality  \eqref{Sharp}. Specifically, if $p\in(1,2)$ then the prefactor $(t-s)^{\frac{1}{p}}$ on the
left-hand side of \eqref{1.5} is now improved by $(t-s)^{\frac{1}{p}-\frac
{1}{2}}$ in \eqref{Sharp}. On the other hand, if $p\in\lbrack2,\infty),$ then
\eqref{Sharp} coincides with the optimal estimate \eqref{Tri}. On the other hand, inequality \eqref{Sharp2}, with $p<2,$ seems to be new.
\end{remark}

We give a positive answer to Open Problem \ref{OpenProblemIntro1}. In fact, we
obtain a stronger version by means of removing the $\theta$-dependance from
the constant $C,$ and working with inhomogeneous norms on both sides of
\eqref{OpenProblemIntro1Eq}. Furthermore, our method works with the more
general family of seminorms formed by $\Vert\cdot\Vert_{\dot{W}%
^{s,p}({\mathbb{R}}^{n}),\alpha}$ (cf. \eqref{DefFractGS}) for arbitrary order of smoothness $\alpha > 0$.

\begin{theorem}
\label{ConjectureBSY} Let $p\in(1,\infty),\,\alpha>0,$ and $0<s<t<\alpha$.
Then there exists a positive constant $C$, independent of $s$ and $t$, such
that
\[
\Vert f\Vert_{L^{p}({\mathbb{R}}^{n})}+\min\{s,\alpha-s\}^{\frac{1}{p}}\Vert
f\Vert_{\dot{W}^{s,p}({\mathbb{R}}^{n}),\alpha}\leq C\Big(\Vert f\Vert
_{L^{p}({\mathbb{R}}^{n})}+\min\{t,\alpha-t\}^{\frac{1}{p}}\Vert f\Vert
_{\dot{W}^{t,p}({\mathbb{R}}^{n}),\alpha}\Big).
\]
In particular, if $\alpha=1$ then (cf. \eqref{FractModuInt})
\[
\Vert f\Vert_{L^{p}({\mathbb{R}}^{n})}+\min\{s,1-s\}^{\frac{1}{p}}\Vert
f\Vert_{\dot{W}^{s,p}({\mathbb{R}}^{n})}\leq C\Big(\Vert f\Vert_{L^{p}%
({\mathbb{R}}^{n})}+\min\{t,1-t\}^{\frac{1}{p}}\Vert f\Vert_{\dot{W}%
^{t,p}({\mathbb{R}}^{n})}\Big)
\]
for all $0<s<t<1$.
\end{theorem}

\section{Interpolation/extrapolation: A primer}

\label{SectionMethod}

The goal of this section is to present our interpolation/extrapolation based
methodology. These methods will be applied in Section \ref{SectionProofs} to prove all the results that were
stated in Section \ref{SectionMR}.

\subsection{Revisiting some classical interpolation formulae}

All the results contained in this subsection are well known. However, for the
purposes of this paper, it is important to establish these results with sharp
control on the dependance with respect to the parameters involved. For the
convenience of the reader, who may not be familiar with interpolation theory,
we give a self-contained exposition, since it is hard to find the material organized in a manner that fits our needs in this paper.

Let $(A_{0},A_{1})$ be a quasi-semi-normed pair. The \emph{real interpolation space}
$(A_{0},A_{1})_{\theta,q},\,\theta\in(0,1),\,q\in(0,\infty]$, is formed by all
the elements $f\in A_{0}+A_{1},$ such that
\begin{equation}
\Vert f\Vert_{(A_{0},A_{1})_{\theta,q}}:=\bigg(\int_{0}^{\infty}\left(
t^{-\theta}K(t,f;A_{0},A_{1})\right)  ^{q}\frac{dt}{t}\bigg)^{1/q}%
<\infty,\label{DefInt}%
\end{equation}
where $K(t,f;A_{0},A_{1})$ denotes the \emph{$K$-functional} for the pair
$(A_{0},A_{1}),$ defined for $t>0,$ by
\[
K(t,f;A_{0},A_{1}):=\inf_{f=f_{0}+f_{1}}\Big(\Vert f_{0}\Vert_{A_{0}}+t\Vert
f_{1}\Vert_{A_{1}}\Big).
\]
We shall simply write $K(t,f)$ when the underlying pair of spaces $(A_0, A_1)$ is
understood from the context. We refer to the standard references on
interpolation theory \cite{BennettSharpley, BerghLofstrom, Triebel}. Sometimes
it is more convenient to work with the modified $K$-functional given by
\[
K_{p}(t,f;A_{0},A_{1}):=K_{p}(t,f)=\inf_{f=f_{0}+f_{1}}\Big(\Vert f_{0}%
\Vert_{A_{0}}^{p}+t^{p}\Vert f_{1}\Vert_{A_{1}}^{p}\Big)^{\frac{1}{p}}%
\]
for $p\in(0,\infty)$. Clearly $K(t,f)\approx K_{p}(t,f)$ with equivalence
constants depending only on $p$. Furthermore, it is plain that
\begin{equation}
(A_{0},A_{1})_{\theta,q}=(A_{1},A_{0})_{1-\theta,q}%
\label{OrderedCouplesCommute}%
\end{equation}
with equality of norms.

Let $A$ be a Banach space and $p\in(0,\infty)$. The vector-valued Lebesgue
spaces $L^{p}({\mathbb{R}}^{n};A)$ are formed by all strongly measurable
functions $f:{\mathbb{R}}^{n}\rightarrow A$ such that
\[
\Vert f\Vert_{L^{p}({\mathbb{R}}^{n};A)}:=\bigg(\int_{{\mathbb{R}}^{n}}\Vert
f(x)\Vert_{A}^{p}\,dx\bigg)^{1/p}<\infty.
\]
The next result is well known.

\begin{lemma}
\label{LemmaInterpolationLpVector} Let $(A_{0},A_{1})$ be a pair of Banach
spaces and let $s\in(0,1),p\in(0,\infty)$. Then
\begin{equation*}
(L^{p}({\mathbb{R}}^{n};A_{0}),L^{p}({\mathbb{R}}^{n};A_{1}))_{s,p}%
=L^{p}({\mathbb{R}}^{n};(A_{0},A_{1})_{s,p}) 
\end{equation*}
and
\begin{equation*}
\Vert f\Vert_{(L^{p}({\mathbb{R}}^{n};A_{0}),L^{p}({\mathbb{R}}^{n}%
;A_{1}))_{s,p}}\approx\Vert f\Vert_{L^{p}({\mathbb{R}}^{n};(A_{0},A_{1}%
)_{s,p})} 
\end{equation*}
where the hidden equivalence constants depend only on $p$.
\end{lemma}

\begin{proof}
It is plain that
\[
K(t,f;L^{p}({\mathbb{R}}^{n};A_{0}),L^{p}({\mathbb{R}}^{n};A_{1}%
))^{p}\approx\int_{{\mathbb{R}}^{n}}K_{p}(t,f(x);A_{0},A_{1})^{p}\,dx.
\]
Therefore using Fubini's theorem,
\begin{align*}
\Vert f\Vert_{(L^{p}({\mathbb{R}}^{n};A_{0}),L^{p}({\mathbb{R}}^{n}%
;A_{1}))_{s,p}}^{p} &  \approx\int_{0}^{\infty}t^{-sp}\int_{{\mathbb{R}}^{n}%
}K(t,f(x);A_{0},A_{1})^{p}\,dx\,\frac{dt}{t}\\
&  \approx\int_{{\mathbb{R}}^{n}}\Vert f(x)\Vert_{(A_{0},A_{1})_{s,p}}%
^{p}\,dx.
\end{align*}

\end{proof}

For $s\in{\mathbb{R}}$ and $p\in(0,\infty)$, the space $\ell_{p}%
^{s}({{\mathbb{Z}}})$ is formed by all the scalar-valued sequences $\xi=(\xi_{j}%
)_{j\in{{\mathbb{Z}}}},$ such that
\[
\Vert\xi\Vert_{\ell_{p}^{s}({{\mathbb{Z}}})}:=\bigg(\sum_{j=-\infty}^{\infty
}2^{jsp}|\xi_{j}|^{p}\bigg)^{1/p}<\infty.
\]
Similarly, we can define the spaces $\ell_{p}^{s}({{\mathbb{N}}}_{0})$ where
${{\mathbb{N}}}_{0}={{\mathbb{N}}}\cup\{0\}$.

\begin{lemma}
\label{Lemma4} Let $-\infty<s_{0}<s_{1}<\infty,\,\theta\in(0,1),s=(1-\theta
)s_{0}+\theta s_{1},$ and let $p,q\in(0,\infty)$. Then
\[
(\ell_{q}^{s_{0}}({{\mathbb{Z}}}),\ell_{q}^{s_{1}}({{\mathbb{Z}}}))_{\theta
,p}=\ell_{p}^{s}({{\mathbb{Z}}}),
\]
with
\[
\bigg(\frac{1}{\theta^{\frac{1}{\max\{p,q\}}}}+\frac{1}{(1-\theta)^{\frac
{1}{\max\{p,q\}}}}\bigg)\Vert\xi\Vert_{\ell_{p}^{s}({{\mathbb{Z}}})}%
\lesssim\Vert\xi\Vert_{(\ell_{q}^{s_{0}}({{\mathbb{Z}}}),\ell_{q}^{s_{1}%
}({{\mathbb{Z}}}))_{\theta,p}}\lesssim\bigg(\frac{1}{\theta^{\frac{1}%
{\min\{p,q\}}}}+\frac{1}{(1-\theta)^{\frac{1}{\min\{p,q\}}}}\bigg)\Vert
\xi\Vert_{\ell_{p}^{s}({{\mathbb{Z}}})}%
\]
uniformly w.r.t. $\theta$. The corresponding result also holds true for
${{\mathbb{N}}}_{0}$-indexed sequences.
\end{lemma}

\begin{proof}

It is well known and easy to see that
\[
K_{q}(t,\xi;\ell_{q}^{s_{0}}({{\mathbb{Z}}}),\ell_{q}^{s_{1}}({{\mathbb{Z}}%
}))=\bigg(\sum_{j=-\infty}^{\infty}[\min\{2^{js_{0}},2^{js_{1}}t\}|\xi
_{j}|]^{q}\bigg)^{1/q}.
\]
By the monotonicity properties of the $K$-functional, the following estimates
hold uniformly w.r.t. $\theta\in(0,1)$,
\begin{align}
\Vert\xi\Vert_{(\ell_{q}^{s_{0}}({{\mathbb{Z}}}),\ell_{q}^{s_{1}}%
({{\mathbb{Z}}}))_{\theta,p}} &  \approx\bigg(\sum_{l=-\infty}^{\infty
}2^{l\theta(s_{1}-s_{0})p}K(2^{-l(s_{1}-s_{0})},\xi;\ell_{q}^{s_{0}%
}({{\mathbb{Z}}}),\ell_{q}^{s_{1}}({{\mathbb{Z}}}))^{p}\bigg)^{1/p}\nonumber\\
&  \approx\bigg(\sum_{l=-\infty}^{\infty}2^{l\theta(s_{1}-s_{0})p}%
\bigg(\sum_{j=-\infty}^{\infty}[2^{js_{0}}\min\{1,2^{(j-l)(s_{1}-s_{0})}%
\}|\xi_{j}|]^{q}\bigg)^{p/q}\bigg)^{1/p}\nonumber\\
&  \approx A+B,\label{A+B}%
\end{align}
where
\[
A:=\bigg(\sum_{l=-\infty}^{\infty}2^{l(\theta-1)(s_{1}-s_{0})p}\bigg(\sum
_{j=-\infty}^{l}2^{js_{1}q}|\xi_{j}|^{q}\bigg)^{p/q}\bigg)^{1/p}%
\]
and
\[
B:=\bigg(\sum_{l=-\infty}^{\infty}2^{l\theta(s_{1}-s_{0})p}\bigg(\sum
_{j=l}^{\infty}2^{js_{0}q}|\xi_{j}|^{q}\bigg)^{p/q}\bigg)^{1/p}.
\]

We consider now two possible cases. Assume first $p\geq q.$
Therefore
\begin{align*}
A &  \geq\bigg(\sum_{l=-\infty}^{\infty}2^{l(\theta-1)(s_{1}-s_{0})p}%
\sum_{j=-\infty}^{l}2^{js_{1}p}|\xi_{j}|^{p}\bigg)^{1/p}\\
&  =\bigg(\sum_{j=-\infty}^{\infty}2^{js_{1}p}|\xi_{j}|^{p}\sum_{l=j}^{\infty
}2^{l(\theta-1)(s_{1}-s_{0})p}\bigg)^{1/p}\\
&  \approx\frac{1}{(1-\theta)^{1/p}}\bigg(\sum_{j=-\infty}^{\infty}2^{jsp}%
|\xi_{j}|^{p}\bigg)^{1/p}.
\end{align*}
By the sharp version of Hardy's inequality (see, e.g., \cite[p. 196]%
{SteinWeiss}),
\begin{align*}
A &  =\bigg(\sum_{l=-\infty}^{\infty}\bigg(2^{l(\theta-1)(s_{1}-s_{0})q}%
\sum_{j=-\infty}^{l}2^{js_{1}q}|\xi_{j}|^{q}\bigg)^{p/q}\bigg)^{1/p}\\
&  \lesssim\frac{1}{(1-\theta)^{1/q}}\bigg(\sum_{l=-\infty}^{\infty}%
2^{lsp}|\xi_{l}|^{p}\bigg)^{1/p}.
\end{align*}
Hence the term $A$ can be estimated by
\begin{equation}
\frac{1}{(1-\theta)^{1/p}}\bigg(\sum_{j=-\infty}^{\infty}2^{jsp}|\xi_{j}%
|^{p}\bigg)^{1/p}\lesssim A\lesssim\frac{1}{(1-\theta)^{1/q}}\bigg(\sum
_{j=-\infty}^{\infty}2^{j s p}|\xi_{j}|^{p}\bigg)^{1/p}.\label{A}%
\end{equation}
Furthermore, similar estimates also hold for $B$, namely,
\begin{equation}
\frac{1}{\theta^{1/p}}\bigg(\sum_{j=-\infty}^{\infty}2^{jsp}|\xi_{j}%
|^{p}\bigg)^{1/p}\lesssim B\lesssim\frac{1}{\theta^{1/q}}\bigg(\sum
_{j=-\infty}^{\infty}2^{jsp}|\xi_{j}|^{p}\bigg)^{1/p}.\label{B}%
\end{equation}
Combining \eqref{A+B}, \eqref{A} and \eqref{B}, we arrive at the desired
result under $p\geq q$.

Suppose now $p<q$. Then
\begin{align*}
A &  \leq\bigg(\sum_{l=-\infty}^{\infty}2^{l(\theta-1)(s_{1}-s_{0})p}%
\sum_{j=-\infty}^{l}2^{js_{1}p}|\xi_{j}|^{p}\bigg)^{1/p}\\
&  \approx\frac{1}{(1-\theta)^{1/p}}\bigg(\sum_{j=-\infty}^{\infty}2^{jsp}%
|\xi_{j}|^{p}\bigg)^{1/p}.
\end{align*}
On the other hand, by a sharp version of Copson's inequality (see, e.g., \cite{Leindler}) we find that
\[
A\gtrsim\frac{1}{(1-\theta)^{1/q}}\bigg(\sum_{j=-\infty}^{\infty}2^{jsp}%
|\xi_{j}|^{p}\bigg)^{1/p}.
\]
Hence
\[
\frac{1}{(1-\theta)^{1/q}}\bigg(\sum_{j=-\infty}^{\infty}2^{jsp}|\xi_{j}%
|^{p}\bigg)^{1/p}\lesssim A\lesssim\frac{1}{(1-\theta)^{1/p}}\bigg(\sum
_{j=-\infty}^{\infty}2^{jsp}|\xi_{j}|^{p}\bigg)^{1/p},
\]
with hidden constants of equivalence uniform w.r.t. $\theta\in(0,1)$. The term
$B$ can be estimated similarly. Combining estimates yields the desired result for $p < q$. 
\end{proof}

We close this subsection by recalling the well-known characterization of
$\dot{B}_{p,q}^{s}({\mathbb{R}}^{n})$ as an interpolation space between
$L^{p}({\mathbb{R}}^{n})$ and $\dot{W}^{k}_{p}({\mathbb{R}}^{n})$; see, e.g.,
\cite[Chapter 5, Corollary 4.13, page 341]{BennettSharpley} and \cite[Theorem
6.3.1, page 147]{BerghLofstrom}.

\begin{lemma}
\label{LemmaIntBesov} Let $s\in(0,1),\,k\in{{\mathbb{N}}}$ and $p\in
(1,\infty)$. Then\footnote{The limiting values $p=1,\infty$ are also
adimissible.}
\[
(L^{p}({\mathbb{R}}^{n}),\dot{W}_{p}^{k}({\mathbb{R}}^{n}))_{s,q}=\dot
{B}_{p,q}^{sk}({\mathbb{R}}^{n}).
\]

\end{lemma}

\subsection{Extrapolation theory}

Generally speaking, the extrapolation theory of Jawerth and Milman
\cite{JawerthMilman} develops methods to recover the endpoint spaces of a
given parametrized scale of quasi-Banach spaces $\{A_{\theta}:\theta
\in(0,1)\}.$ In particular, it contains a careful analysis of the asymptotic
behavior of the related quasi-norms $\Vert\cdot\Vert_{A_{\theta}}$ as
$\theta\rightarrow0^{+}$ or $\theta\rightarrow1^{-}$ . In  \cite{JawerthMilman} it
is shown that $(\theta(1-\theta)p)^{\frac{1}{p}}$ is the appropriate
normalization factor for the real interpolation method $(A_{0},A_{1}%
)_{\theta,p}$ (cf. \eqref{DefInt}). This assertion is well-illustrated by the
following results.

\begin{lemma}
\label{LemmaInterpolation12} Let $p \in(0, \infty)$. Then there exists $C >
0$, which depends only on $p$, such that, for every $\theta\in(0, 1)$ and $f
\in A_{0} \cap A_{1}$,
\[
\|f\|_{(A_{0}, A_{1})_{\theta, p}} \leq C \Big( \frac{1}{\theta^{1/p}}
\|f\|_{A_{0}} + \frac{1}{(1-\theta)^{1/p}} \|f\|_{A_{1}} \Big).
\]

\end{lemma}

\begin{proof}
Since $K(t,f)\leq\min\{\Vert f\Vert_{A_{0}},t\Vert f\Vert_{A_{1}}\}$, it
follows that
\begin{align*}
\Vert f\Vert_{(A_{0},A_{1})_{\theta,p}}^{p} &  =\int_{0}^{1}t^{-\theta
p}K(t,f)^{p}\frac{dt}{t}+\int_{1}^{\infty}t^{-\theta p}K(t,f)^{p}\frac{dt}%
{t}\\
&  \leq\Vert f\Vert_{A_{1}}^{p}\int_{0}^{1}t^{(1-\theta)p}\frac{dt}{t}+\Vert
f\Vert_{A_{0}}^{p}\int_{1}^{\infty}t^{-\theta p}\frac{dt}{t}\\
&  =\Vert f\Vert_{A_{1}}^{p}\frac{1}{(1-\theta)p}+\Vert f\Vert_{A_{0}}%
^{p}\frac{1}{\theta p}.
\end{align*}

\end{proof}

\begin{lemma}
\label{LemmaExtrapolSharp}

\begin{enumerate}
[\upshape(i)]

\item Assume $\theta\in(0, 1)$ and $0 < p \leq q < \infty$. Then
\[
(\theta(1-\theta) p)^{1/p} (A_{0}, A_{1})_{\theta, p} \hookrightarrow
(\theta(1-\theta) q)^{1/q} (A_{0}, A_{1})_{\theta, q}.
\]

\item Assume further that $(A_{0},A_{1})$ is ordered (i.e., $A_{1}%
\hookrightarrow A_{0}$). Let $0<\theta<\eta<1$ and let $p,q\in(0,\infty)$.
Then, the norm of the embedding\footnote{The assumption that the couple
$(A_{0},A_{1})$ is ordered is not restrictive, i.e., a corresponding embedding
result still holds for general couples $(A_{0},A_{1})$. However, this
embedding is simplified when dealing with ordered couples, which will be the
only case of interest in this paper.}
\[
(A_{0},A_{1})_{\eta,p}\hookrightarrow(A_{0},A_{1})_{\theta,q}%
\]
does not exceed
\[
\left\{
\begin{array}
[c]{cl}%
\frac{(\eta(1-\eta))^{1/p}}{(\theta(1-\theta))^{1/q}} & \text{if}\quad p\leq
q,\\
& \\
\frac{1}{(\eta-\theta)^{1/q-1/p}}+\frac{\eta^{1/p}(1-\eta)^{1/p}}{\theta
^{1/q}} & \text{if}\quad p>q,
\end{array}
\right.
\]
uniformly w.r.t. $\eta$ and $\theta$.
\end{enumerate}
\end{lemma}

\begin{proof}
The formula that was stated in (i) can be found in \cite[(3.2), p.
19]{JawerthMilman}. Concerning (ii) with $p=q$, i.e.,
\begin{equation}
(\eta(1-\eta))^{1/p}(A_{0},A_{1})_{\eta,p}\hookrightarrow(\theta
(1-\theta))^{1/p}(A_{0},A_{1})_{\theta,p},\label{KMOrder}%
\end{equation}
we refer to \cite[Corollary 2.1]{KaradzhovMilman}. Furthermore the case $p\leq
q$ in (ii) can be obtained as a direct consequence of (i) and \eqref{KMOrder}.

Next we concentrate on the case $p > q$ in (ii). Since $A_{1} \hookrightarrow
A_{0}$, it is plain to see that $K(t, f) \approx\|f\|_{A_{0}}$ for $t > 1$.
Then
\begin{equation}
\label{LemmaEmbConstantInterpolationProof1}\|f\|_{(A_{0}, A_{1})_{\theta, q}}
\approx\bigg(\int_{0}^{1} (t^{-\theta} K(t, f))^{q} \frac{dt}{t} \bigg)^{1/q}
+ \frac{1}{\theta^{1/q}} \|f\|_{A_{0}}.
\end{equation}

On the one hand, applying H\"{o}lder's inequality, we can estimate the
integral given on the right-hand side of
\eqref{LemmaEmbConstantInterpolationProof1} by
\begin{equation}
\bigg(\int_{0}^{1}(t^{-\theta}K(t,f))^{q}\frac{dt}{t}\bigg)^{1/q}\lesssim
\frac{1}{(\eta-\theta)^{1/q-1/p}}\bigg(\int_{0}^{1}(t^{-\eta}K(t,f))^{p}%
\frac{dt}{t}\bigg)^{1/p}. \label{LemmaEmbConstantInterpolationProof2}%
\end{equation}
On the other hand, using monotonicity properties of the $K$-functional, the
second term in the right-hand side of
\eqref{LemmaEmbConstantInterpolationProof1} can be dominated as follows
\begin{equation}
\Vert f\Vert_{A_{0}}\lesssim\min\{\eta,1-\eta\}^{1/p}\bigg(\int_{0}^{\infty
}(t^{-\eta}K(t,f))^{p}\frac{dt}{t}\bigg)^{1/p}.
\label{LemmaEmbConstantInterpolationProof3}%
\end{equation}
Inserting \eqref{LemmaEmbConstantInterpolationProof2} and
\eqref{LemmaEmbConstantInterpolationProof3} into
\eqref{LemmaEmbConstantInterpolationProof1}, we get
\[
\Vert f\Vert_{(A_{0},A_{1})_{\theta,q}}\lesssim\bigg(\frac{1}{(\eta
-\theta)^{1/q-1/p}}+\frac{\min\{\eta,1-\eta\}^{1/p}}{\theta^{1/q}}\bigg)\Vert
f\Vert_{(A_{0},A_{1})_{\eta,p}}.
\]

\end{proof}

The previous result tells us that the standard interpolation norm $\Vert
\cdot\Vert_{(A_{0},A_{1})_{\theta,p}}$ given in \eqref{DefInt} can be
renormalized by $(\theta(1-\theta)p)^{1/p}\Vert\cdot\Vert_{(A_{0}%
,A_{1})_{\theta,p}}$ so that sharp estimates w.r.t. the interpolation
parameter $\theta$ can now be achieved. Another important example of this
phenomenon occurs in reiteration formulas. Indeed, we recall the following
sharp versions of reiteration properties for the real method.

\begin{lemma}
[{\cite[Theorems 2.11 and 2.12]{KaradzhovMilman} and \cite[Theorem
3]{KaradzhovMilmanXiao}}]\label{LemmaKMX2} Let $s_{0}, s_{1}, \theta\in(0, 1)$
and $p, q \in(0, \infty)$.

\begin{enumerate}
[\upshape(i)]

\item Let $s=(1-\theta)s_{0}+\theta s_{1}$. The following embeddings hold
uniformly w.r.t. $\theta$
\begin{align*}
(\theta(1-\theta))^{-1/\min\{p,q\}}(A_{0},A_{1})_{s,p}  &  \hookrightarrow
((A_{0},A_{1})_{s_{0},q},(A_{0},A_{1})_{s_{1},q})_{\theta,p}\\
&  \hspace{-4cm}\hookrightarrow(\theta(1-\theta))^{-1/\max\{p,q\}}(A_{0}%
,A_{1})_{s,p}.
\end{align*}

\item The following embeddings hold uniformly w.r.t. $s_{1}$ and $\theta$
\begin{align*}
s_{1}^{1/p-1/q} [(1-s_{1})^{-1/q} + (1-\theta)^{-1/\min\{p, q\}}] (A_{0},
A_{1})_{\theta s_{1}, p}  & \\
&  \hspace{-8cm}\hookrightarrow(A_{0}, (A_{0}, A_{1})_{s_{1}, q})_{\theta, p}
\hookrightarrow s_{1}^{1/p} (1-\theta)^{-1/\max\{p, q\}} (A_{0}, A_{1})_{
\theta s_{1}, p}.
\end{align*}

\end{enumerate}
\end{lemma}

A key role in our arguments is played by the continuity properties of general
interpolation scales obtained in \cite{Milman} (see also
\cite{KaradzhovMilmanXiao}). In particular, we will make use of the following
result for the real interpolation method.

\begin{theorem}
\label{teomarkao}Let $p>0$ and $f\in A_{0}\cap A_{1}$. Then
\begin{equation}
\lim_{\theta\rightarrow1^{-}}(\theta(1-\theta)p)^{\frac{1}{p}}\Vert
f\Vert_{(A_{0},A_{1})_{\theta,p}}=\sup_{t>0}\frac{K(t,f;A_{0},A_{1})}{t}
\label{ProofFractGS1}%
\end{equation}
and
\begin{equation}
\lim_{\theta\rightarrow0^{+}}(\theta(1-\theta)p)^{\frac{1}{p}}\Vert
f\Vert_{(A_{0},A_{1})_{\theta,p}}=\sup_{t>0}K(t,f;A_{0},A_{1}).
\label{ProofFractGS1New}%
\end{equation}

\end{theorem}

\subsection{Butzer seminorms via interpolation}

A key result in our approach is that the family of semi-norms $\Vert\cdot
\Vert_{\dot{W}^{s,p}({\mathbb{R}}^{n}),t}$ introduced in \eqref{DefFractGS}
can be generated (with sharp constants) via interpolation of the classical
pair $(L^{p}({\mathbb{R}}^{n}),\dot{H}^{t,p}({\mathbb{R}}^{n}))$. This is the
assertion contained in the following

\begin{lemma}
\label{LemmaIntFractGS} Let $\theta\in(0,1),t\in(0,\infty)$ and $p\in
(1,\infty)$. Then
\begin{equation}
\Vert f\Vert_{(L^{p}({\mathbb{R}}^{n}),\dot{H}^{t,p}({\mathbb{R}}%
^{n}))_{\theta,p}}\approx\Vert f\Vert_{\dot{W}^{\theta t,p}({\mathbb{R}}%
^{n}),t} \label{LemmaIntFractGSEq1}%
\end{equation}
and
\begin{equation}
\Vert f\Vert_{(L^{p}({\mathbb{R}}^{n}),H^{t,p}({\mathbb{R}}^{n}))_{\theta,p}%
}\approx \Vert f\Vert_{W^{\theta t,p}({\mathbb{R}}^{n}),t}
\label{LemmaIntFractGSEq2}%
\end{equation}
where hidden equivalence constants are independent of $\theta$ (but depending
on $n,p$ and $t$).
\end{lemma}

\begin{remark}
Letting $t=1$ in the previous result (cf. \eqref{SobRiesz} and
\eqref{FractModuInt}), we have
\begin{equation}
\Vert f\Vert_{(L^{p}({\mathbb{R}}^{n}),\dot{W}_{p}^{1}({\mathbb{R}}%
^{n}))_{\theta,p}}\approx\Vert f\Vert_{\dot{W}^{\theta,p}({\mathbb{R}}^{n})}
\label{IntSob1}%
\end{equation}
and, cf. \eqref{DefFractGSIn},
\[
\Vert f\Vert_{(L^{p}({\mathbb{R}}^{n}),W_{p}^{1}({\mathbb{R}}^{n}))_{\theta
,p}}\approx\frac{1}{(\theta(1-\theta)p)^{\frac{1}{p}}}\Vert f\Vert
_{L^{p}({\mathbb{R}}^{n})}+\Vert f\Vert_{\dot{W}^{\theta,p}({\mathbb{R}}^{n}%
)}.
\]
In particular, \eqref{IntSob1} was already shown in \cite[Lemma 1]{Milman}.
\end{remark}

\begin{proof}
[Proof of Lemma \ref{LemmaIntFractGS}]We shall employ the characterizations of
the corresponding $K$-functional given by
\begin{equation}
K(u^{t},f;L^{p}({\mathbb{R}}^{n}),\dot{H}^{t,p}({\mathbb{R}}^{n}))\approx
\frac{1}{u^{n}}\int_{|h|<u}\Vert\Delta_{h}^{t}f\Vert_{L^{p}({\mathbb{R}}^{n}%
)}\,dh\approx\sup_{|h|<u}\Vert\Delta_{h}^{t}f\Vert_{L^{p}({\mathbb{R}}^{n})}
\label{ProofFractGS1*}%
\end{equation}
for $u>0$; this characterization is well known in the classical case
$t=k\in{{\mathbb{N}}}$ (see, e.g., \cite[Chapter 5, (4.42), p. 341]%
{BennettSharpley} and \cite[Theorem 6.7.3]{BerghLofstrom}), for the general
case $t\in(0,\infty)$ we refer to \cite[(1.10)]{Kolomoitsev} (with \cite{Butzer} and
\cite{Wilmes} as a forerunners).

Inserting the first equivalence in \eqref{ProofFractGS1*} into the definition
of interpolation space (cf. \eqref{DefInt}), we have
\begin{equation*}
\Vert f\Vert_{(L^{p}({\mathbb{R}}^{n}),\dot{H}^{t,p}({\mathbb{R}}%
^{n}))_{\theta,p}}\approx\bigg(\int_{0}^{\infty}u^{-\theta tp-np}%
\bigg(\int_{|h|<u}\Vert\Delta_{h}^{t}f\Vert_{L^{p}({\mathbb{R}}^{n}%
)}\,dh\bigg)^{p}\frac{du}{u}\bigg)^{1/p}.
\end{equation*}
Consequently, the desired result \eqref{LemmaIntFractGSEq1} would be
established once we are able to prove (cf. \eqref{DefFractGS})
\begin{equation}
\Vert f\Vert_{\dot{W}^{\theta t,p}({\mathbb{R}}^{n}),t}\approx\bigg(\int%
_{0}^{\infty}u^{-\theta tp-np}\bigg(\int_{|h|<u}\Vert\Delta_{h}^{t}%
f\Vert_{L^{p}({\mathbb{R}}^{n})}\,dh\bigg)^{p}\frac{du}{u}\bigg)^{1/p}%
\label{ClaimDefFractGS}%
\end{equation}
uniformly w.r.t. $\theta\in(0,1)$.

To deal with the estimate $\gtrsim$ in \eqref{ClaimDefFractGS}, we apply
H\"{o}lder's inequality and change the order of integration so that
\begin{align*}
\int_{0}^{\infty}u^{-\theta tp-np}\bigg(\int_{|h|<u}\Vert\Delta_{h}^{t}%
f\Vert_{L^{p}({\mathbb{R}}^{n})}\,dh\bigg)^{p}\frac{du}{u} &  \lesssim\int%
_{0}^{\infty}u^{-\theta tp-n}\int_{|h|<u}\Vert\Delta_{h}^{t}f\Vert
_{L^{p}({\mathbb{R}}^{n})}^{p}\,dh\frac{du}{u}\\
&  \hspace{-6cm}\approx\int_{{\mathbb{R}}^{n}}|h|^{-\theta tp-n}\Vert
\Delta_{h}^{t}f\Vert_{L^{p}({\mathbb{R}}^{n})}^{p}\,dh=\Vert f\Vert_{\dot
{W}^{\theta t,p}({\mathbb{R}}^{n}),t}^{p}.
\end{align*}
Concerning the estimate $\lesssim$ in \eqref{ClaimDefFractGS}, we use basic
monotonicity properties and the second estimate given in
\eqref{ProofFractGS1*} to get
\begin{align*}
\Vert f\Vert_{\dot{W}^{\theta t,p}({\mathbb{R}}^{n}),t}^{p} &  =\int%
_{{\mathbb{R}}^{n}}\frac{\Vert\Delta_{h}^{t}f\Vert_{L^{p}({\mathbb{R}}^{n}%
)}^{p}}{|h|^{\theta tp+n}}\,dh\\
&  =\sum_{j=-\infty}^{\infty}\int_{2^{j-1}<|h|\leq2^{j}}\frac{\Vert\Delta
_{h}^{t}f\Vert_{L^{p}({\mathbb{R}}^{n})}^{p}}{|h|^{\theta tp+n}}\,dh\\
&  \lesssim\sum_{j=-\infty}^{\infty}2^{-j\theta tp}\sup_{|h|\leq2^{j}}%
\Vert\Delta_{h}^{t}f\Vert_{L^{p}({\mathbb{R}}^{n})}^{p}\\
&  \approx\sum_{j=-\infty}^{\infty}2^{-j\theta tp}\bigg(\frac{1}{2^{j n}}%
\int_{|h|<2^{j}}\Vert\Delta_{h}^{t}f\Vert_{L^{p}({\mathbb{R}}^{n}%
)}\,dh\bigg)^{p}.
\end{align*}

Next we focus on \eqref{LemmaIntFractGSEq2}. The $K$-functional for the pair
$(L^{p}({\mathbb{R}}^{n}),H^{t,p}({\mathbb{R}}^{n}))$ can be estimated as (see
\cite[(4.2)]{Wilmes})
\[
K(u,f;L^{p}({\mathbb{R}}^{n}),H^{t,p}({\mathbb{R}}^{n}))\approx\min
\{1,u\}\Vert f\Vert_{L^{p}({\mathbb{R}}^{n})}+\sup_{|h|\leq u^{\frac{1}{t}}%
}\Vert\Delta_{h}^{t}f\Vert_{L^{p}({\mathbb{R}}^{n})}%
\]
for $u>0$ and $f\in L^{p}({\mathbb{R}}^{n})$. Therefore, by
\eqref{ProofFractGS1*},
\[
K(u,f;L^{p}({\mathbb{R}}^{n}),H^{t,p}({\mathbb{R}}^{n}))\approx\min
\{1,u\}\Vert f\Vert_{L^{p}({\mathbb{R}}^{n})}+K(u,f;L^{p}({\mathbb{R}}%
^{n}),\dot{H}^{t,p}({\mathbb{R}}^{n})).
\]
By elementary computations we find
\begin{align*}
\Vert f\Vert_{(L^{p}({\mathbb{R}}^{n}),H^{t,p}({\mathbb{R}}^{n}))_{\theta,p}%
}^{p} &  \approx\bigg(\int_{0}^{\infty}u^{-\theta p}\min\{1,u\}^{p}\frac
{du}{u}\bigg)\Vert f\Vert_{L^{p}({\mathbb{R}}^{n})}^{p}+\Vert f\Vert
_{(L^{p}({\mathbb{R}}^{n}),\dot{H}^{t,p}({\mathbb{R}}^{n}))_{\theta,p}}^{p}\\
&  =\frac{1}{\theta(1-\theta)p}\Vert f\Vert_{L^{p}({\mathbb{R}}^{n})}%
^{p}+\Vert f\Vert_{(L^{p}({\mathbb{R}}^{n}),\dot{H}^{t,p}({\mathbb{R}}%
^{n}))_{\theta,p}}^{p}.
\end{align*}
Therefore, by \eqref{LemmaIntFractGSEq1},
\[
\Vert f\Vert_{(L^{p}({\mathbb{R}}^{n}),H^{t,p}({\mathbb{R}}^{n}))_{\theta,p}%
}^{p}\approx\frac{1}{\theta(1-\theta)p}\Vert f\Vert_{L^{p}({\mathbb{R}}^{n}%
)}^{p}+\Vert f\Vert_{\dot{W}^{\theta t,p}({\mathbb{R}}^{n}),t}^{p}.
\]

\end{proof}

The following result shows that the seminorm $\Vert f\Vert_{\dot{W}%
^{s,p}({\mathbb{R}}^{n}),t}$ is equivalent to the classical Gagliardo seminorm
$\Vert f\Vert_{\dot{W}^{s,p}({\mathbb{R}}^{n})}$, but the constants of
equivalence blow up as $s\rightarrow t^{-}$.

\begin{lemma}
\label{LemmaFractInt} Let $0<s<t<1$ and $1<p<\infty$. Then
\begin{equation}
\frac{1}{(t-s)^{\frac{1}{\max\{p,2\}}}}\Vert f\Vert_{\dot{W}^{s,p}({\mathbb{R}}^{n}%
)}\lesssim\Vert f\Vert_{\dot{W}^{s,p}({\mathbb{R}}^{n}),t}\lesssim\frac
{1}{(t-s)^{\frac{1}{\min\{p,2\}}}}\Vert f\Vert_{\dot{W}^{s,p}({\mathbb{R}}^{n})},
\label{LemmaFractInt1}%
\end{equation}
where the hidden equivalence constants are independent of $s$.
\end{lemma}

\begin{proof}
Recall the relationships between $\dot{H}^{t,p}({\mathbb{R}}^{n})$ and
$\dot{B}_{p,q}^{t}({\mathbb{R}}^{n})$ (cf. Lemma \ref{TableCoincidences}),
\[
\dot{B}_{p,\min\{p,2\}}^{t}({\mathbb{R}}^{n})\hookrightarrow\dot{H}%
^{t,p}({\mathbb{R}}^{n})\hookrightarrow\dot{B}_{p,\max\{p,2\}}^{t}%
({\mathbb{R}}^{n}),
\]
therefore, by interpolation, the following embeddings hold uniformly w.r.t.
$\theta\in(0,1)$
\begin{equation}
(L^{p}({\mathbb{R}}^{n}),\dot{B}_{p,\min\{p,2\}}^{t}({\mathbb{R}}%
^{n}))_{\theta,p}\hookrightarrow(L^{p}({\mathbb{R}}^{n}),\dot{H}%
^{t,p}({\mathbb{R}}^{n}))_{\theta,p}\hookrightarrow(L^{p}({\mathbb{R}}%
^{n}),\dot{B}_{p,\max\{p,2\}}^{t}({\mathbb{R}}^{n}))_{\theta,p}%
.\label{LemmaFractInt1Proof}%
\end{equation}
Furthermore, by Lemma \ref{LemmaIntBesov}, we can rewrite
\[
(L^{p}({\mathbb{R}}^{n}),\dot{B}_{p,q}^{t}({\mathbb{R}}^{n}))_{\theta
,p}=(L^{p}({\mathbb{R}}^{n}),(L^{p}({\mathbb{R}}^{n}),\dot{W}_{p}%
^{1}({\mathbb{R}}^{n}))_{t,q})_{\theta,p},\qquad q\in(0,\infty),
\]
with related equivalence constants independent of $\theta$. Combining this and
Lemma \ref{LemmaKMX2}(ii), we obtain the embeddings
\begin{align*}
(1-\theta)^{-1/\min\{p,q\}}(L^{p}({\mathbb{R}}^{n}),\dot{W}_{p}^{1}%
({\mathbb{R}}^{n}))_{\theta t,p} &  \hookrightarrow(L^{p}({\mathbb{R}}%
^{n}),\dot{B}_{p,q}^{t}({\mathbb{R}}^{n}))_{\theta,p}\\
&  \hspace{-4cm}\hookrightarrow(1-\theta)^{-1/\max\{p,q\}}(L^{p}({\mathbb{R}%
}^{n}),\dot{W}_{p}^{1}({\mathbb{R}}^{n}))_{\theta t,p}.%
\end{align*}
In light of Lemma \ref{LemmaIntFractGS} (see also \eqref{IntSob1}), the
previous embeddings turn out to be equivalent to
\begin{equation}
(1-\theta)^{-1/\max\{p,q\}} \|f\|_{\dot{W}^{\theta
t,p}({\mathbb{R}}^{n})} \lesssim \|f\|_{(L^{p}({\mathbb{R}}^{n}),\dot{B}_{p,q}^{t}({\mathbb{R}}%
^{n}))_{\theta,p} } \lesssim (1-\theta)^{-1/\min\{p,q\}} \|f\|_{\dot{W}^{\theta t,p}({\mathbb{R}}^{n})}.\label{LemmaFractInt2}%
\end{equation}
Combining \eqref{LemmaFractInt1Proof} and \eqref{LemmaFractInt2}, we find%
$$
(1-\theta)^{-1/\max\{p,2\}} \|f\|_{\dot{W}^{\theta
t,p}({\mathbb{R}}^{n})} \lesssim \|f\|_{(L^{p}({\mathbb{R}}^{n}),\dot{H}^{t,p}({\mathbb{R}}%
^{n}))_{\theta,p}} \lesssim (1-\theta)^{-1/\min\{p,2\}} \|f\|_{\dot{W}^{\theta t,p}({\mathbb{R}}^{n})}.
$$
Equivalently (cf. \eqref{LemmaIntFractGSEq1})
\[
(1-\theta)^{-1/\max\{p,2\}}\Vert f\Vert_{\dot{W}^{\theta t,p}({\mathbb{R}}%
^{n})}\lesssim\Vert f\Vert_{\dot{W}^{\theta t,p}({\mathbb{R}}^{n}),t}%
\lesssim(1-\theta)^{-1/\min\{p,2\}}\Vert f\Vert_{\dot{W}^{\theta
t,p}({\mathbb{R}}^{n})}
\]
which, after the change of variables $\theta\leftrightarrow\frac{s}{t},$
yields the desired result \eqref{LemmaFractInt1}.
\end{proof}

\section{Proofs of the main results}

\label{SectionProofs}

\subsection{Proof of Proposition \ref{PropQuestion1.11}}

We now recall some results that we need from the theory of Fourier series with
monotonic coefficients.

Suppose that the Fourier series of $f\in L^{1}({\mathbb{T}})$ is given by%
\begin{equation}
f(x)\sim\sum_{\nu=1}^{\infty}c_{\nu}\cos(\nu x),\qquad x\in{\mathbb{T}%
},\label{GenFS}%
\end{equation}
with
\begin{equation}
c_{\nu}\geq c_{\nu+1}\geq\cdots \geq 0\qquad\text{and}\qquad c_{\nu}\rightarrow
0.\label{MonCond}%
\end{equation}
A well-known theorem of Hardy--Littlewood (see, e.g., \cite[p. 129, Vol.
II]{Zygmund}) asserts that $f\in L^{p}({\mathbb{T}})$ if and only if
\[
\sum_{\nu=1}^{\infty}\nu^{p-2}c_{\nu}^{p}<\infty.
\]
Furthermore
\begin{equation}
\Vert f\Vert_{L^{p}({\mathbb{T}})}^{p}\approx\sum_{\nu=1}^{\infty}\nu
^{p-2}c_{\nu}^{p}.\label{HLtheorem}%
\end{equation}
This result has been further extended in \cite{DominguezTikhonov} in order to
deal with other spaces of smooth functions (namely, Besov spaces and Sobolev
spaces), as well as more general classes of Fourier series with monotone-type
coefficients (the so called, general monotone class). In particular, it was
shown in \cite[Theorem 4.25]{DominguezTikhonov} that if $f$ is given by
\eqref{GenFS} with coefficients satisfying \eqref{MonCond} then it belongs to
$H^{t,p}({\mathbb{T}}),\,t\in{\mathbb{R}},\,p\in(1,\infty),$ if and only if
\[
\sum_{\nu=1}^{\infty}\nu^{tp+p-2}c_{\nu}^{p}<\infty.
\]
Moreover,
\begin{equation}
\Vert f\Vert_{H^{t,p}({\mathbb{T}})}^{p}\approx\sum_{\nu=1}^{\infty}%
\nu^{tp+p-2}c_{\nu}^{p}.\label{HLSob}%
\end{equation}

Let $t > 0$ and assume $c_{\nu}=\nu^{-t-1+\frac{1}{p}},\,\nu\in{{\mathbb{N}}}$. By virtue of
\eqref{HLtheorem} and \eqref{HLSob}
\begin{equation}
\Vert f\Vert_{L^{p}({\mathbb{T}})}^{p}\approx\sum_{\nu=1}^{\infty}\nu^{p-2}%
\nu^{-tp-p+1}=\sum_{\nu=1}^{\infty}\nu^{-tp-1}<\infty\label{Counter1*}%
\end{equation}
and
\begin{equation*}
\Vert f\Vert_{H^{t,p}({\mathbb{T}})}^{p}\approx\sum_{\nu=1}^{\infty}%
\nu^{tp+p-2}\nu^{-tp-p+1}=\sum_{\nu=1}^{\infty}\nu^{-1}=\infty,
\end{equation*}
respectively. Consequently, the Fourier series $f$ defined by
\eqref{GenFS} belongs to $L^{p}({\mathbb{T}})$ but not to
$H^{t,p}({\mathbb{T}})$.

To conclude, it remains to show that
\begin{equation}
\Vert f\Vert_{\dot{W}^{s_{k},p}({\mathbb{T}})}\leq C(t-s_{k})^{-\frac{1}{p}%
}\label{ClaimCounterEx}%
\end{equation}
uniformly w.r.t. $k\in{{\mathbb{N}}}$. To proceed, we argue as follows. In
view of \eqref{IntSob1},
\begin{equation}
\Vert f\Vert_{\dot{W}^{s_{k},p}({\mathbb{T}})}\approx\bigg(\int_{0}^{\infty
}u^{-s_{k}p}K(u,f;L^{p}({\mathbb{T}}),\dot{W}_{p}^{1}({\mathbb{T}}))^{p}%
\frac{du}{u}\bigg)^{\frac{1}{p}},\label{Counter1}%
\end{equation}
with hidden constants of equivalence independent of $k$. Next we
estimate $K(u,f;L^{p}({\mathbb{T}}),\dot{W}_{p}^{1}({\mathbb{T}}))$.

For $u>0$, the classical moduli of smoothness of $f\in L^{p}({\mathbb{T}})$ is
defined by
\begin{equation}
\omega(f,u)_{p}:=\sup_{|h|\leq u}\Vert\Delta_{h}f\Vert_{L^{p}({\mathbb{T}}%
)}.\label{DefModuli}%
\end{equation}
Assume further that $f$ satisfies \eqref{GenFS} and \eqref{MonCond}. Then it
is known that $\omega(f,u)_{p}$ can be estimated in terms of the Fourier
coefficients of $f$. Namely,
\begin{equation}
\omega(f,u)_{p}\approx\bigg(\sum_{\nu=1}^{\infty}\min\{1,\nu u\}^{p}\nu
^{p-2}c_{\nu}\bigg)^{\frac{1}{p}},\qquad u\in(0,1);\label{GTGM}%
\end{equation}
(cf. \cite[Theorem 6.2]{Tikhonov} and \cite[Theorem 6.1]{GorbachevTikhonov}
where this is actually done even in the more general context of general
monotone Fourier coefficients.)

Specializing \eqref{GTGM} to $f$ given by \eqref{GenFS}, yields
\[
\omega(f,l^{-1})_{p}\approx l^{-1}\bigg(\sum_{\nu=1}^{l}\nu^{(1-t)p-1}%
\bigg)^{\frac{1}{p}}+\bigg(\sum_{\nu=l}^{\infty}\nu^{-tp-1}\bigg)^{\frac{1}%
{p}}\approx\frac{1}{(t(1-t))^{\frac{1}{p}}}l^{-t}%
\]
for every $l\in{{\mathbb{N}}}$ and $t\in(0,1)$. According to
\eqref{ProofFractGS1*} (with $t=1$) and \eqref{DefModuli}, the previous estimate is equivalent to
\begin{equation}
K(l^{-1},f;L^{p}({\mathbb{T}}),\dot{W}_{p}^{1}({\mathbb{T}}))\approx\frac
{1}{(t(1-t))^{\frac{1}{p}}}l^{-t}.\label{Counter2}%
\end{equation}
Furthermore, we have the trivial estimate
\begin{equation}
K(u,f;L^{p}({\mathbb{T}}),\dot{W}_{p}^{1}({\mathbb{T}}))\leq\Vert
f\Vert_{L^{p}({\mathbb{T}})}.\label{Counter3}%
\end{equation}

Putting together \eqref{Counter1}, \eqref{Counter2}, \eqref{Counter3},
applying monotonicity properties of $K$-functionals, and taking into account
that $s_{k}\uparrow t$, we obtain
\begin{align*}
\Vert f\Vert_{\dot{W}^{s_{k},p}({\mathbb{T}})}^{p} &  \lesssim\int_{0}%
^{1}u^{-s_{k}p}K(u,f;L^{p}({\mathbb{T}}),\dot{W}_{p}^{1}({\mathbb{T}}%
))^{p}\frac{du}{u}+\frac{1}{s_{k}p}\Vert f\Vert_{L^{p}({\mathbb{T}})}^{p}\\
&  \approx\sum_{l=1}^{\infty}l^{s_{k}p-1}K(l^{-1},f;L^{p}({\mathbb{T}}%
),\dot{W}_{p}^{1}({\mathbb{T}}))^{p}+\Vert f\Vert_{L^{p}({\mathbb{T}})}^{p}\\
&  \approx\sum_{l=1}^{\infty}l^{-(t-s_{k})p-1}+\Vert f\Vert_{L^{p}%
({\mathbb{T}})}^{p}\\
&  \approx\frac{1}{t-s_{k}}+\Vert f\Vert_{L^{p}({\mathbb{T}})}^{p}\\
&  \lesssim\frac{1}{t-s_{k}}.
\end{align*}
This concludes the proof of \eqref{ClaimCounterEx}. \qed

\begin{remark}
\label{RemarkCounterexample} Similar ideas can be used to establish the
counterpart of Proposition \ref{PropQuestion1.11} in ${\mathbb{R}}^{n}$. We
will only sketch the proof. Consider the monotone function
\[
F_{0}(u)=\left\{
\begin{array}
[c]{cl}%
u^{-t-n+\frac{n}{p}} & \text{if}\quad u>1,\\
& \\
1 & \text{if}\quad u\in(0,1],
\end{array}
\right.
\]
and define the radial function $f(x)=f_{0}(|x|),\,x\in{\mathbb{R}}^{n},$ where
$f_{0}$ is the inverse Fourier--Hankel transform of $F_{0}$, i.e.,
\[
f_{0}(u)=\frac{2}{\Gamma\big(\frac{n}{2}\big)(2\sqrt{\pi})^{n}}\int%
_{0}^{\infty}F_{0}(\xi)j_{n/2-1}(u\xi)\xi^{n-1}\,d\xi,
\]
where $j_{\alpha}(u)=\Gamma(\alpha+1)(u/2)^{-\alpha}J_{\alpha}(u)$ is the
normalized Bessel function ($j_{\alpha}(0)=1$), $\alpha\geq-1/2$, and
$J_{\alpha}$ is the classical Bessel function of the first kind of order
$\alpha$. Applying the analogue of Hardy--Littlewood theorem for
$L^{p}({\mathbb{R}}^{n}),\,p>\frac{2 n}{n+1},$ (cf. \cite[Theorem
1]{GorbachevLiflyandTikhonov} and \cite[(4.10)]{GorbachevTikhonov}), we have
\[
\Vert f\Vert_{L^{p}({\mathbb{R}}^{n})}^{p}\approx\int_{0}^{\infty}%
u^{n p-n-1}F_{0}(u)^{p}\,du=\int_{0}^{1}u^{n p-n-1}\,du+\int_{1}^{\infty
}u^{-tp-1}\,du<\infty,
\]
which gives the Euclidean setting counterpart of \eqref{Counter1*}. On the
other hand, using the counterpart of \eqref{HLSob} for $H^{t,p}({\mathbb{R}%
}^{n})$, which may be found in \cite[Theorem 4.8]{DominguezTikhonov}, we have
\begin{align*}
\Vert f\Vert_{H^{t,p}({\mathbb{R}}^{n})}^{p} &  \approx\int_{0}^{1}%
u^{n p-n-1}F_{0}^{p}(u)\,du+\int_{1}^{\infty}u^{tp+n p-n-1}F_{0}(u)^{p}\,du\\
&  =\int_{0}^{1}u^{n p-n-1}\,du+\int_{1}^{\infty}\frac{du}{u}=\infty.
\end{align*}
The analog of \eqref{GTGM} for the moduli of smoothness $\omega(f,u)_{p}%
=\sup_{|h|\leq u}\Vert\Delta_{h}f\Vert_{L^{p}({\mathbb{R}}^{n})}$ (cf.
\eqref{DefModuli}) was obtained in \cite[Corollary 4.1 and (7.6)]%
{GorbachevTikhonov}, namely,
\[
\omega(f,u)_{p}^{p}\approx u^{p}\int_{0}^{1/u}\xi^{p+n p-n-1}F_{0}^{p}%
(\xi)\,d\xi+\int_{1/u}^{\infty}\xi^{n p-n-1}F_{0}^{p}(\xi)\,d\xi.
\]
In particular (since $t\in(0,1)$),
\[
\omega(f,u)_{p}\approx u^{t}\qquad\text{for}\qquad u\in(0,1).
\]
Following now line by line the arguments given in the proof of Proposition
\ref{PropQuestion1.11}, one can show that there exists $C>0$ such that
\[
\Vert f\Vert_{\dot{W}^{s_{k},p}({\mathbb{R}}^{n})}\leq C(t-s_{k})^{-\frac
{1}{p}}\qquad\text{for every}\qquad k\in{{\mathbb{N}}}.
\]

\end{remark}

\subsection{Proof of Theorem \ref{Theorem343Fract}}

We give a unified proof in the inhomogeneous setting (i.e., \eqref{Aux61} and
\eqref{Aux62}). It follows from Lemma \ref{LemmaIntFractGS} that
\[
(s(t-s))^{\frac{1}{p}}\Vert f\Vert_{(L^{p}({\mathbb{R}}^{n}),H^{t,p}%
({\mathbb{R}}^{n}))_{\frac{s}{t},p}}\approx\Vert f\Vert_{L^{p}({\mathbb{R}%
}^{n})}+(s(t-s))^{\frac{1}{p}}\Vert f\Vert_{\dot{W}^{s,p}({\mathbb{R}}^{n}%
),t}.
\]
Consequently, both \eqref{Aux61} \& \eqref{Aux62} will be established if we
are able to show that
\begin{equation}
\Vert f\Vert_{F_{p,2}^{\bar{r}}({\mathbb{R}}^{n})}\lesssim(s(t-s))^{\frac
{1}{p}}\Vert f\Vert_{(L^{p}({\mathbb{R}}^{n}),H^{t,p}({\mathbb{R}}%
^{n}))_{\frac{s}{t},p}}.\label{51}%
\end{equation}
To prove this inequality, we invoke the method of retractions (cf.
\cite{Triebel}) which shows that the interpolation space $(L^{p}({\mathbb{R}%
}^{n}),H^{t,p}({\mathbb{R}}^{n}))_{\frac{s}{t},p}$ can be isomorphically
identified (with constants of equivalence independent of $s$) with
$(L^{p}({\mathbb{R}}^{n};\ell_{2}({{\mathbb{N}_0}})),L^{p}({\mathbb{R}}^{n}%
;\ell_{2}^{t}({{\mathbb{N}_0}})))_{\frac{s}{t},p}$. In view of the interpolation
formula given in Lemma \ref{LemmaInterpolationLpVector},
\[
\Vert f\Vert_{(L^{p}({\mathbb{R}}^{n};\ell_{2}({{\mathbb{N}_0}})),L^{p}%
({\mathbb{R}}^{n};\ell_{2}^{t}({{\mathbb{N}_0}})))_{\frac{s}{t},p}}\approx\Vert
f\Vert_{L^{p}({\mathbb{R}}^{n};(\ell_{2}({{\mathbb{N}_0}}),\ell_{2}%
^{t}({{\mathbb{N}_0}}))_{\frac{s}{t},p})},
\]
which reduces\footnote{Here we must take into account that $F_{p,2}^{\bar{r}%
}({\mathbb{R}}^{n})$ can be identified with the vector-valued space
$L^{p}({\mathbb{R}}^{n};\ell_{2}^{\bar{r}}({{\mathbb{N}_0}}))$. } \eqref{51} to
the following sharp embedding at the level of sequence spaces
\begin{equation}
(s(t-s))^{\frac{1}{p}}(\ell_{2}({{\mathbb{N}_0}}),\ell_{2}^{t}({{\mathbb{N}_0}%
}))_{\frac{s}{t},p}\hookrightarrow\ell_{2}^{\bar{r}}({{\mathbb{N}_0}%
}).\label{ClaimSeqSpaces}%
\end{equation}
Since (cf. Lemma \ref{Lemma4})
\begin{equation*}
(\bar{r}(t-\bar{r}))^{\frac{1}{2}}(\ell_{2}({{\mathbb{N}_0}}),\ell_{2}%
^{t}({{\mathbb{N}_0}}))_{\frac{\bar{r}}{t},2}=\ell_{2}^{\bar{r}}({{\mathbb{N}_0}%
}),
\end{equation*}
the embedding \eqref{ClaimSeqSpaces} turns out to be equivalent to
\begin{equation}
(s(t-s))^{\frac{1}{p}}(\ell_{2}({{\mathbb{N}_0}}),\ell_{2}^{t}({{\mathbb{N}_0}%
}))_{\frac{s}{t},p}\hookrightarrow(\bar{r}(t-\bar{r}))^{\frac{1}{2}}(\ell
_{2}({{\mathbb{N}_0}}),\ell_{2}^{t}({{\mathbb{N}_0}}))_{\frac{\bar{r}}{t}%
,2}.\label{ClaimSeqSpacesNew}%
\end{equation}

It remains to show the validity of \eqref{ClaimSeqSpacesNew}. To do this, we
distinguish two possible cases. Suppose first that $p\leq2$, then by Lemma
\ref{LemmaExtrapolSharp}(ii) (note that $\bar{r}<s$),
\begin{equation*}
(s(t-s))^{\frac{1}{p}}(\ell_{2}({{\mathbb{N}_0}}),\ell_{2}^{t}({{\mathbb{N}_0}%
}))_{\frac{s}{t},p}\hookrightarrow(\bar{r}(t-\bar{r}))^{\frac{1}{2}}(\ell
_{2}({{\mathbb{N}_0}}),\ell_{2}^{t}({{\mathbb{N}_0}}))_{\frac{\bar{r}}{t},2}.
\end{equation*}

Suppose now that $p>2$. Applying again Lemma \ref{LemmaExtrapolSharp}(ii), we
have
\begin{equation}
\bigg(\frac{1}{(s-\bar{r})^{1/2-1/p}}+\frac{s^{1/p}(t-s)^{1/p}}{\bar{r}^{1/2}%
}\bigg)(\ell_{2}({{\mathbb{N}_0}}),\ell_{2}^{t}({{\mathbb{N}_0}}))_{\frac{s}{t}%
,p}\hookrightarrow(\ell_{2}({{\mathbb{N}_0}}),\ell_{2}^{t}({{\mathbb{N}_0}%
}))_{\frac{\bar{r}}{t},2}.\label{G2}%
\end{equation}
Furthermore, it is easy to check that under the assumptions (1) \& (2) in Theorem \ref{Theorem343Fract}, the
constant in \eqref{G2} behaves like
\begin{equation*}
\frac{1}{(s-\bar{r})^{1/2-1/p}}+\frac{s^{1/p}(t-s)^{1/p}}{\bar{r}^{1/2}%
}\approx(s(t-s))^{\frac{1}{p}}(\bar{r}(t-\bar{r}))^{-\frac{1}{2}%
},
\end{equation*}
which yields the desired embedding \eqref{ClaimSeqSpacesNew}.

For the homogeneous setting (i.e., \eqref{Aux6} and \eqref{Aux6new}), observe
that it is enough to deal with $r=\bar{r}$, since for $r\in(0,\bar{r})$ we
have
\[
\Vert f\Vert_{\dot{F}_{p,2}^{r}({\mathbb{R}}^{n})}\leq\Vert f\Vert_{\dot
{F}_{p,2}^{0}({\mathbb{R}}^{n})}+\Vert f\Vert_{\dot{F}_{p,2}^{\bar{r}%
}({\mathbb{R}}^{n})}\leq C(\Vert f\Vert_{L^{p}({\mathbb{R}}^{n})}+\Vert
f\Vert_{\dot{F}_{p,2}^{\bar{r}}({\mathbb{R}}^{n})}).
\]
Henceforth, we concentrate on \eqref{Aux6} and \eqref{Aux6new} with $r=\bar
{r}$. The methodology for the inhomogeneous setting given above (i.e.,
\eqref{Aux61} and \eqref{Aux62}) yields (cf. \eqref{51})
\begin{equation}
\Vert f\Vert_{\dot{F}_{p,2}^{\bar{r}}({\mathbb{R}}^{n})}\lesssim
(s(t-s))^{\frac{1}{p}}\Vert f\Vert_{(L^{p}({\mathbb{R}}^{n}),H^{t,p}%
({\mathbb{R}}^{n}))_{\frac{s}{t},p}}.\label{FrHom}%
\end{equation}
Indeed, in this case the analog of \eqref{ClaimSeqSpaces} reads
\begin{equation}
(s(t-s))^{\frac{1}{p}}(\ell_{2}({{\mathbb{Z}}}),\ell_{2}^{t}({{\mathbb{Z}}%
}))_{\frac{s}{t},p}\hookrightarrow\ell_{2}^{\bar{r}}({{\mathbb{Z}}%
})\label{EmbNonOrder}%
\end{equation}
where
\[
\bar{r}=\left\{
\begin{array}
[c]{cl}%
t-\Lambda(t-s), & \text{if}\quad s\rightarrow t^{-},\\
& \\
\frac{1}{\Lambda}s, & \text{if}\quad s\rightarrow0^{+}.
\end{array}
\right.
\]
Here the couple $(\ell_{2}({{\mathbb{Z}}}),\ell_{2}^{t}({{\mathbb{Z}}}))$ is
not ordered and we cannot invoke Lemma \ref{LemmaExtrapolSharp}(ii) directly
to deal with \eqref{EmbNonOrder}. However, this difficulty can be overcome as
follows. We write ${{\mathbb{Z}}}={{\mathbb{N}}}_{0}\cup{{\mathbb{N}}}_{-}$
where ${{\mathbb{N}}}_{-}:=\{-j:j\in{{\mathbb{N}}}\}$. Since $\ell_{2}%
^{t}({{\mathbb{N}}}_{0})\hookrightarrow\ell_{2}({{\mathbb{N}}}_{0})$ and
$\ell_{2}({{\mathbb{N}}}_{-})\hookrightarrow\ell_{2}^{t}({{\mathbb{N}}}_{-})$,
it follows from \eqref{ClaimSeqSpaces} that
\begin{equation}
(s(t-s))^{\frac{1}{p}}(\ell_{2}({{\mathbb{N}}}_{0}),\ell_{2}^{t}({{\mathbb{N}%
}}_{0}))_{\frac{s}{t},p}\hookrightarrow\ell_{2}^{\bar{r}}({{\mathbb{N}}}%
_{0})\label{Decomp1}%
\end{equation}
and, by \eqref{OrderedCouplesCommute}, and a simple change of variables,
\begin{equation}
(s(t-s))^{\frac{1}{p}}(\ell_{2}({{\mathbb{N}}}_{-}),\ell_{2}^{t}({{\mathbb{N}%
}}_{-}))_{\frac{s}{t},p}\hookrightarrow\ell_{2}^{\bar{r}}({{\mathbb{N}}}%
_{-}).\label{Decomp2}%
\end{equation}
Let $\xi=(\xi)_{j\in{{\mathbb{Z}}}}$ and consider the related sequences
$\xi^{1}=(\xi_{j})_{j\in{{\mathbb{N}}}_{0}}$ and $\xi^{2}=(\xi_{j}%
)_{j\in{{\mathbb{N}}}_{-}}$. By triangle inequality and \eqref{Decomp1},
\eqref{Decomp2}, we derive
\begin{align*}
\Vert\xi\Vert_{\ell_{2}^{\bar{r}}({{\mathbb{Z}}})} &  \leq\Vert\xi^{1}%
\Vert_{\ell_{2}^{\bar{r}}({{\mathbb{N}}}_{0})}+\Vert\xi^{2}\Vert_{\ell
_{2}^{\bar{r}}({{\mathbb{N}}}_{-})}\\
&  \leq C(s(t-s))^{\frac{1}{p}}(\Vert\xi^{1}\Vert_{(\ell_{2}({{\mathbb{N}}%
}_{0}),\ell_{2}^{t}({{\mathbb{N}}}_{0}))_{\frac{s}{t},p}}+\Vert\xi^{2}%
\Vert_{(\ell_{2}({{\mathbb{N}}}_{-}),\ell_{2}^{t}({{\mathbb{N}}}_{-}%
))_{\frac{s}{t},p}})\\
&  \leq2C(s(t-s))^{\frac{1}{p}}\Vert\xi\Vert_{(\ell_{2}({{\mathbb{Z}}}%
),\ell_{2}^{t}({{\mathbb{Z}}}))_{\frac{s}{t},p}},
\end{align*}
where the last step follows immediately from the interpolation property
applied to the canonical projections $\xi\mapsto\xi^{i},\,i=1,2$. The proof of
\eqref{EmbNonOrder} is complete establishing that \eqref{FrHom} holds. Now the
rest of the proof follows line by line the arguments provided for the
inhomogeneous case. \qed

\subsection{Proof of Theorem \ref{ThemQuestion1.11}}

Without loss of generality, we may assume that $t-s_{k}\leq\frac{t}{4}$.
According to \eqref{Aux6} there is $C=C(n,p,t)>0$ such that, for every $r\leq
t-2(t-s_{k})$,
\[
\Vert f_{k}\Vert_{\dot{F}_{p,2}^{r}({\mathbb{R}}^{n})}\leq C\Big(\Vert
f_{k}\Vert_{L^{p}({\mathbb{R}}^{n})}+(t-s_{k})^{\frac{1}{p}}\Vert f_{k}%
\Vert_{\dot{W}^{s_{k},p}({\mathbb{R}}^{n}),t}\Big)\leq C\Lambda.
\]
Taking limits as $k\rightarrow\infty,$ and noting that $\lim_{k\rightarrow
\infty}t-2(t-s_{k})=t$, the previous estimate yields
\[
\limsup_{k\rightarrow\infty}\Big(\Vert f_{k}\Vert_{L^{p}({\mathbb{R}}^{n}%
)}+\Vert f_{k}\Vert_{\dot{F}_{p,2}^{r}({\mathbb{R}}^{n})}\Big)\lesssim
\Lambda\qquad\text{for all}\qquad r\in(0,t).
\]
Therefore, one can apply \cite[Lemma 2.6]{Braz} to derive $f\in H^{t,p}%
({\mathbb{R}}^{n})$ and
\[
\Vert(-\Delta)^{\frac{t}{2}}f\Vert_{L^{p}({\mathbb{R}}^{n})}\lesssim\Lambda.
\]

\qed

\begin{remark}
It may be instructive to revisit the counterexample provided in Proposition
\ref{PropQuestion1.11} to show the important role played by the seminorms
$\Vert\cdot\Vert_{\dot{W}^{s,p}({\mathbb{R}}^{n}),t}$ in Theorem
\ref{ThemQuestion1.11}. Consider the Fourier series $f$ defined by
\eqref{FourierSeries}. It was shown in Proposition \ref{PropQuestion1.11} that
the condition
\[
\sup_{k\in{{\mathbb{N}}}}\,(t-s_{k})^{\frac{1}{p}}\Vert f\Vert_{\dot{W}%
^{s_{k},p}({\mathbb{T}})}<\infty
\]
is satisfied but $f\not \in H^{t,p}({\mathbb{T}})$. Next we check that this
example does not contradict Theorem \ref{ThemQuestion1.11} since
\begin{equation}
\limsup_{k\rightarrow\infty}\,(t-s_{k})^{\frac{1}{p}}\Vert f\Vert_{\dot
{W}^{s_{k},p}({\mathbb{T}}),t}=\infty. \label{ClaimCounter}%
\end{equation}
Indeed, applying \cite[Theorem 6.2]{Tikhonov} one can estimate, for every
$l\in{{\mathbb{N}}}$,
\begin{align*}
K(l^{-t},f;L^{p}({\mathbb{T}}),\dot{H}^{t,p}({\mathbb{T}}))  &  \approx
l^{-t}\bigg(\sum_{\nu=1}^{l}\nu^{tp+p-2}\nu^{-tp-p+1}\bigg)^{1/p}%
+\bigg(\sum_{\nu=l}^{\infty}\nu^{p-2}\nu^{-tp-p+1}\bigg)^{1/p}\\
&  \approx l^{-t}(1+\log l)^{1/p}%
\end{align*}
and thus, by Lemma \ref{LemmaIntFractGS},
\begin{align*}
\Vert f\Vert_{\dot{W}^{s_{k},p}({\mathbb{T}}),t}^{p}  &  \approx\Vert
f\Vert_{(L^{p}({\mathbb{T}}),\dot{H}^{t,p}({\mathbb{T}}))_{\frac{s_{k}}{t},p}%
}^{p}\\
&  \gtrsim\sum_{l=0}^{\infty}2^{ls_{k}p}K(2^{-lt},f;L^{p}({\mathbb{T}}%
),\dot{H}^{t,p}({\mathbb{T}}))^{p}\\
&  \approx\sum_{l=0}^{\infty}2^{-l(t-s_{k})p}(1+l)\\
&  \gtrsim(t-s_{k})^{-1}\sum_{l=\left\lfloor {\frac{1}{t-s_{k}}}\right\rfloor
}^{\infty}2^{-l(t-s_{k})p}\\
&  \approx\frac{(t-s_{k})^{-1}}{1-2^{-(t-s_{k})p}}\approx(t-s_{k})^{-2}.
\end{align*}
Hence
\[
(t-s_{k})^{1/p}\Vert f\Vert_{\dot{W}^{s_{k},p}({\mathbb{T}}),t}\gtrsim
(t-s_{k})^{-1/p}%
\]
which yields \eqref{ClaimCounter} (since $s_{k}\uparrow t$).
\end{remark}

\subsection{Proof of Theorem \ref{TheoremFBBM}}

Let $(A_{0},A_{1})=(L^{p}({\mathbb{R}}^{n}),\dot{H}^{t,p}({\mathbb{R}}^{n}))$.
Recall that for this couple for $f\in$ $H
^{t,p}({\mathbb{R}}^{n}),$ (cf. \cite[Corollary 10, page 75]{Wilmes},
\cite[Section 1.3, Property 1]{Kolomoitsev})
\[
\sup_{t>0}\frac{K(t,f;L^{p}({\mathbb{R}}^{n}),\dot{H}^{t,p}({\mathbb{R}}%
^{n}))}{t}\approx\Vert(-\Delta)^{\frac{t}{2}}f\Vert_{L^{p}({\mathbb{R}}^{n})}%
\]
and
\[
\sup_{t>0}K(t,f;L^{p}({\mathbb{R}}^{n}),\dot{H}^{t,p}({\mathbb{R}}%
^{n}))\approx\Vert f\Vert_{L^{p}({\mathbb{R}}^{n})}.
\]
Therefore \eqref{FractGSAssertion} (respectively, \eqref{FractGSAssertionMS})
is an immediate consequence of Lemma \ref{LemmaIntFractGS} and
\eqref{ProofFractGS1} (respectively, \eqref{ProofFractGS1New}).

Concerning the inhomogeneous setting (i.e., \eqref{FractGSAssertionIn} and \eqref{FractGSAssertionMSIn}), the proofs are immediate consequences of their homogeneous counterparts (i.e., \eqref{FractGSAssertion} and \eqref{FractGSAssertionMS}) applied to \eqref{DefFractGSIn}.

\qed

\subsection{Proof of Theorem \ref{ThmSobolevBSY}}

Assume $0<r<t<1$. Applying Lemma \ref{LemmaInterpolation12} to the
Triebel--Lizorkin pair $(A_{0},A_{1})=(\dot{F}_{p,2}^{r}({\mathbb{R}}%
^{n}),\dot{F}_{p,2}^{t}({\mathbb{R}}^{n}))$, one has
\[
\Vert f\Vert_{(\dot{F}_{p,2}^{r}({\mathbb{R}}^{n}),\dot{F}_{p,2}%
^{t}({\mathbb{R}}^{n}))_{\theta,p}}\lesssim\frac{1}{\theta^{1/p}}\Vert
f\Vert_{\dot{F}_{p,2}^{r}({\mathbb{R}}^{n})}+\frac{1}{(1-\theta)^{1/p}}\Vert
f\Vert_{\dot{F}_{p,2}^{t}({\mathbb{R}}^{n})}%
\]
for every $\theta\in(0,1)$. Since (cf. Lemma \ref{TableCoincidences}(i))
\[
\dot{F}_{p,2}^{r}({\mathbb{R}}^{n})\hookrightarrow\dot{B}_{p,\max\{p,2\}}%
^{r}({\mathbb{R}}^{n})\qquad\text{and}\qquad\dot{F}_{p,2}^{t}({\mathbb{R}}%
^{n})\hookrightarrow\dot{B}_{p,\max\{p,2\}}^{t}({\mathbb{R}}^{n}),
\]
the previous inequality implies
\begin{equation}
\Vert f\Vert_{(\dot{B}_{p,\max\{p,2\}}^{r}({\mathbb{R}}^{n}),\dot{B}%
_{p,\max\{p,2\}}^{t}({\mathbb{R}}^{n}))_{\theta,p}}\lesssim\frac{1}%
{\theta^{1/p}}\Vert f\Vert_{\dot{F}_{p,2}^{r}({\mathbb{R}}^{n})}+\frac
{1}{(1-\theta)^{1/p}}\Vert f\Vert_{\dot{F}_{p,2}^{t}({\mathbb{R}}^{n})}.
\label{ProofThmSobolevBSY1}%
\end{equation}
Next we compute the interpolation norm given in the left-hand side of
\eqref{ProofThmSobolevBSY1}. Indeed, since $r,t\in(0,1)$, we can invoke Lemma
\ref{LemmaIntBesov} to state
\[
\dot{B}_{p,\max\{p,2\}}^{r}({\mathbb{R}}^{n})=(L^{p}({\mathbb{R}}^{n}),\dot
{W}_{p}^{1}({\mathbb{R}}^{n}))_{r,\max\{p,2\}}%
\]
and
\[
\dot{B}_{p,\max\{p,2\}}^{t}({\mathbb{R}}^{n})=(L^{p}({\mathbb{R}}^{n}),\dot
{W}_{p}^{1}({\mathbb{R}}^{n}))_{t,\max\{p,2\}}.
\]
Therefore, in light of Lemmas \ref{LemmaKMX2}(i) and \eqref{IntSob1}, we get,
uniformly w.r.t. $\theta$,
\begin{align*}
\Vert f\Vert_{(\dot{B}_{p,\max\{p,2\}}^{r}({\mathbb{R}}^{n}),\dot{B}%
_{p,\max\{p,2\}}^{t}({\mathbb{R}}^{n}))_{\theta,p}}  &  \approx\Vert
f\Vert_{((L^{p}({\mathbb{R}}^{n}),\dot{W}_{p}^{1}({\mathbb{R}}^{n}%
))_{r,\max\{p,2\}},(L^{p}({\mathbb{R}}^{n}),\dot{W}_{p}^{1}({\mathbb{R}}%
^{n}))_{t,\max\{p,2\}})_{\theta,p}}\\
&  \hspace{-4cm}\gtrsim(\theta(1-\theta))^{-1/\max\{p,2\}}\Vert f\Vert
_{(L^{p}({\mathbb{R}}^{n}),\dot{W}_{p}^{1}({\mathbb{R}}^{n}))_{(1-\theta
)r+\theta t,p}}\\
&  \hspace{-4cm}\approx(\theta(1-\theta))^{-1/\max\{p,2\}}\Vert f\Vert
_{\dot{W}^{(1-\theta)r+\theta t,p}({\mathbb{R}}^{n})}.
\end{align*}
Combining this and \eqref{ProofThmSobolevBSY1}, we arrive at
\begin{equation}
\frac{1}{(\theta(1-\theta))^{1/\max\{p,2\}}}\Vert f\Vert_{\dot{W}%
^{(1-\theta)r+\theta t,p}({\mathbb{R}}^{n})}\lesssim\frac{1}{\theta^{1/p}%
}\Vert f\Vert_{\dot{F}_{p,2}^{r}({\mathbb{R}}^{n})}+\frac{1}{(1-\theta)^{1/p}%
}\Vert f\Vert_{\dot{F}_{p,2}^{t}({\mathbb{R}}^{n})}.
\label{ProofThmSobolevBSY2}%
\end{equation}
Given $s\in(r,t)$, we choose $\theta\in(0,1)$ such that $s=(1-\theta)r+\theta
t$ and make the corresponding change of variables in
\eqref{ProofThmSobolevBSY2}, i.e.,
\[
\frac{1}{((s-r)(t-s))^{1/\max\{p,2\}}}\Vert f\Vert_{\dot{W}^{s,p}({\mathbb{R}%
}^{n})}\lesssim\frac{1}{(s-r)^{1/p}}\Vert f\Vert_{\dot{F}_{p,2}^{r}%
({\mathbb{R}}^{n})}+\frac{1}{(t-s)^{1/p}}\Vert f\Vert_{\dot{F}_{p,2}%
^{t}({\mathbb{R}}^{n})}.
\]

The proofs in the limiting cases $r=0$ and $t=1$ (i.e., \eqref{WLD} and \eqref{WLD2} respectively) are easier and we
omit further details. \qed

\subsection{Proof of Theorem \ref{ConjectureBSY}}

Since $H^{\alpha,p}({\mathbb{R}}^{n})\hookrightarrow L^{p}({\mathbb{R}}^{n})$,
it follows from Lemma \ref{LemmaExtrapolSharp}(ii) that
\[
(s(\alpha-s))^{1/p}\Vert f\Vert_{(L^{p}({\mathbb{R}}^{n}),H^{\alpha
,p}({\mathbb{R}}^{n}))_{\frac{s}{\alpha},p}}\lesssim(t(\alpha-t))^{1/p}\Vert
f\Vert_{(L^{p}({\mathbb{R}}^{n}),H^{\alpha,p}({\mathbb{R}}^{n}))_{\frac
{t}{\alpha},p}}.
\]
Combining with \eqref{LemmaIntFractGSEq2} yields
\begin{align*}
\Vert f\Vert_{L^{p}({\mathbb{R}}^{n})}+(s(\alpha-s))^{1/p}\Vert f\Vert
_{\dot{W}^{s,p}({\mathbb{R}}^{n}),\alpha} &  \approx (s(1-s))^{1/p}\Vert
f\Vert_{(L^{p}({\mathbb{R}}^{n}),H^{\alpha,p}({\mathbb{R}}^{n}))_{\frac
{s}{\alpha},p}}\\
&  \hspace{-5cm}\lesssim(t(1-t))^{1/p}\Vert f\Vert_{(L^{p}({\mathbb{R}}%
^{n}),H^{\alpha,p}({\mathbb{R}}^{n}))_{\frac{t}{\alpha},p}}\\
&  \hspace{-5cm}\approx\Vert f\Vert_{L^{p}({\mathbb{R}}^{n})}+(t(\alpha
-t))^{1/p}\Vert f\Vert_{\dot{W}^{t,p}({\mathbb{R}}^{n}),\alpha}.
\end{align*}
The desired result now follows since $s(\alpha-s)\approx\min\{s,\alpha-s\}$. \qed

\section*{Appendix A: Sharp versions of fractional Sobolev inequalities}

For $p \in (0, \infty)$ and $q \in (0, \infty]$, the \emph{Lorentz space} $L^{p, q}(\R^n)$ is formed by all measurable functions $f$ defined on $\R^n$ such that
$$
	\|f\|_{L^{p, q}(\R^n)} := \bigg(\int_0^\infty \big(u^{\frac{1}{p}} f^*(u) \big)^q \frac{du}{u} \bigg)^{\frac{1}{q}} < \infty
$$ 
(with the usual modification if $q=\infty$). 
As usual, $f^*$ denotes the non-increasing rearrangement of $f$. By $f^{**}$ we denote the maximal function given by $f^{**}(u) := \frac{1}{u} \int_0^{u} f^*(v) \, dv$.  We refer to \cite{BennettSharpley} and \cite{BerghLofstrom} for detailed accounts on Lorentz spaces.

Let $0 < s < 1, \, 1 \leq p < \frac{n}{s}$ and $p^* = \frac{n p}{n- s p}$. The Bourgain--Brezis--Mironescu--Maz'ya--Shasposhnikova formula (cf. \cite{BBM1} and \cite{Mazya}) claims that there exists $C= C(n, p) > 0$ such that
\begin{equation}\label{BBMnew}
	\|f\|^p_{L^{p^{*}, p}(\R^n)} \leq C \frac{s (1-s)}{(n- s p)^{p}} \|f\|_{\dot{W}^{s, p}(\R^n)}^p;
\end{equation}
see also \cite{KaradzhovMilmanXiao}. 
This inequality can be considered as a sharp version of the classical Sobolev embedding
\begin{equation}\label{Sob}
	\dot{W}^1_p(\R^n) \hookrightarrow L^{\frac{n p}{n- p}, p}(\R^n), \qquad 1 \leq p < n.
\end{equation}
Indeed, \eqref{Sob} follows from \eqref{BBMnew} by taking limits as $s \to 1^{-}$ (cf. \eqref{BBM}). 

Let $t \in (0, 1)$ be fixed. Note that the formula \eqref{BBMnew} does not provide any insight if $s \to t^{-}$. So the natural question here is: can we obtain an analogue of \eqref{BBMnew} for fractional smoothness $t$? Or equivalently, what is the fractional counterpart of \eqref{BBMnew} related to the classical embedding
\begin{equation}\label{HLSFract}
	\dot{H}^{t, p}(\R^n) \hookrightarrow L^{\frac{n p}{n- t p}, p}(\R^n), \qquad 1 < p < \frac{n}{t}?
\end{equation}
The answer is given again in terms of the Butzer seminorms $\|\cdot\|_{\dot{W}^{s, p}(\R^n), t}$ (cf. \eqref{DefFractGS}). 

\begin{theorem}\label{ThmBBMt}
	Let $0 < s < t \leq 1, 1 < p < \frac{n}{s}$ and $p^{*} = \frac{n p}{n-s p}$. Assume $f \in W^{s, p}(\R^n)$. Then there exists $C = C(n, p, t) > 0$ such that
	\begin{equation}\label{BBMFract}
	\|f\|^p_{L^{p^{*}, p}(\R^n)} \leq C \frac{s (t-s)}{(n- s p)^{p}} \|f\|_{\dot{W}^{s, p}(\R^n), t}^p.
	\end{equation}
	In particular, if $t=1$ then one recovers \eqref{BBMnew}. 
\end{theorem}

\begin{remark}
 Let $t > 0$ and $1 < p < \frac{n}{t}$. Taking limits as $s \to t^{-}$ in \eqref{BBMFract} and applying Theorem \ref{TheoremFBBM}, we recover \eqref{HLSFract}. 
\end{remark}

\begin{proof}[Proof of Theorem \ref{ThmBBMt}]
	\textsc{Case 1}: Assume $t < \frac{n}{p}$ and $s \in (0, t)$. Under these assumptions, the Hardy--Littlewood--Sobolev theorem (cf. \cite[p. 139]{Oneil}) asserts 
	$$
		\dot{H}^{t, p}(\R^n) \hookrightarrow L^{\frac{n p}{n- t p}, p}(\R^n).
	$$
	If we interpolate this embedding with the trivial one $L^p(\R^n) \hookrightarrow L^p(\R^n)$, we arrive at
	\begin{equation}\label{ForClaim-1}
	\|f\|_{ (L^p(\R^n), L^{\frac{n p}{n- t p}, p}(\R^n))_{\frac{s}{t}, p}} \lesssim	\|f\|_{(L^p(\R^n), \dot{H}^{t, p}(\R^n))_{\frac{s}{t}, p}}  
	\end{equation}
	uniformly with respect to $s \in (0, t)$. In light of Lemma \ref{LemmaIntFractGS},
	\begin{equation}\label{ForClaim0}
		\|f\|_{(L^p(\R^n), \dot{H}^{t, p}(\R^n))_{\frac{s}{t}, p}}  \approx \|f\|_{\dot{W}^{s, p}(\R^n), t}.
	\end{equation}
	Next we show that
	\begin{equation}\label{ForClaim}
		\|f\|_{L^{p^{*}, p}(\R^n)} \approx (s (t-s))^{\frac{1}{p}} \|f\|_{ (L^p(\R^n), L^{\frac{n p}{n- t p}, p}(\R^n))_{\frac{s}{t}, p}}.
	\end{equation}
	Indeed, since $L^{\frac{n p}{n- t p}, p}(\R^n) = (L^p(\R^n), L^\infty(\R^n))_{\frac{t p}{n}, p}$ (see, e.g., \cite[Theorem 5.2.1, p. 109]{BerghLofstrom}), we can write
	$$
		(L^p(\R^n), L^{\frac{n p}{n- t p}, p}(\R^n))_{\frac{s}{t}, p} = (L^p(\R^n), (L^p(\R^n), L^\infty(\R^n))_{\frac{t p}{n}, p})_{\frac{s}{t}, p}
	$$
	with hidden constants of equivalence independent of $s$. Invoking now the sharp reiteration formula given in Lemma \ref{LemmaKMX2}(ii), we get
	\begin{equation}\label{For1}
			(L^p(\R^n), L^{\frac{n p}{n- t p}, p}(\R^n))_{\frac{s}{t}, p} = (t-s)^{-\frac{1}{p}} (L^p(\R^n), L^\infty(\R^n))_{\frac{s p}{n}, p}.
	\end{equation}
	Furthermore, using the well-known characterization of the $K$-functional (cf. \cite[Theorem 5.2.1, p. 109]{BerghLofstrom})
	$$
	K(u, f; L^p(\R^n), L^\infty(\R^n)) \approx \bigg(\int_0^{u^p} (f^*(v))^p \, dv \bigg)^{\frac{1}{p}}
	$$
	and applying Fubini's theorem, we obtain
	\begin{align*}
		\|f\|_{(L^p(\R^n), L^\infty(\R^n))_{\frac{s p}{n}, p}}^p & = \int_0^\infty ( u^{-\frac{s p}{n}} K(u, f; L^p(\R^n), L^\infty(\R^n)))^p \frac{du}{u} \\
		&\approx  \int_0^\infty u^{-\frac{s p}{n}} \int_0^{u} (f^*(v))^p \, dv \frac{du}{u} \\
		& =  \int_0^\infty (f^*(v))^p  \int_v^\infty u^{-\frac{s p}{n}} \frac{du}{u} dv \\
		& = \frac{n}{s p} \|f\|_{L^{p^*, p}(\R^n)}^p.
	\end{align*} 
	Inserting this into \eqref{For1}, we arrive at the desired estimate \eqref{ForClaim}. 
	
	Combining now \eqref{ForClaim-1}--\eqref{ForClaim} we achieve \eqref{BBMFract}. 
	\\
	
	\textsc{Case 2:} Assume that either\footnote{In particular, this argument with $t= \frac{n}{p} \in \N$ completes the proof of \cite[Theorem 8]{KaradzhovMilmanXiao}.}
	$$t > \frac{n}{p} \quad \text{and} \quad s \in \Big(0, \frac{n}{p} \Big)$$
	or
	$$
		t= \frac{n}{p} \quad \text{and} \quad s \to 0^{+}.
	$$
	
%	Since
%	$$
%		\dot{B}^{\frac{N}{p}}_{p, 1}(\R^N) \hookrightarrow L^\infty(\R^N)
%	$$
%	we have
%	$$
%		(L^p(\R^N), \dot{B}^{\frac{N}{p}}_{p, 1}(\R^N))_{\frac{s p}{N}, p} \hookrightarrow (L^p(\R^N), L^\infty(\R^N))_{\frac{s p}{N}, p}.
%	$$
%	
	
	We will make use of the well-known embedding (see, e.g., \cite[p. 164]{Stein})
	$$	
		\dot{H}^{\frac{n}{p}, p} (\R^n) \hookrightarrow \text{BMO}(\R^n)
	$$
	where $\text{BMO}(\R^n)$ is the space of bounded mean oscillation functions of John--Nirenberg \cite{JN} endowed with the seminorm
	$$
		\|f\|_{\text{BMO}(\R^n)} := \|f^{\#}\|_{L^\infty(\R^n)}
	$$
	and $f^{\#}$ is the sharp maximal function of $f \in L^1_{\text{loc}}(\R^n)$, i.e.,
	$$
		f^{\#} (x) := \sup_{Q \ni x} \frac{1}{|Q|} \int_Q |f(x)-f_Q| \, dx, \qquad x \in \R^n,
	$$
	where the supremum runs over all cubes $Q$ in $\R^n$ with sides parallel to the axes of coordinates and $f_Q = \frac{1}{|Q|} \int_Q f$. 
	
	In this case, the analog of \eqref{ForClaim-1} reads as
	\begin{equation}\label{ForClaim-1new}
	\|f\|_{ (L^p(\R^n), \text{BMO}(\R^n))_{\frac{s p}{n}, p}} \lesssim	\|f\|_{(L^p(\R^n), \dot{H}^{\frac{n}{p}, p}(\R^n))_{\frac{s p}{n}, p}}  
	\end{equation}
	for all $s \in (0, \frac{n}{p})$.
	
	We make the following claim, uniformly with respect to $s \in (0, \frac{n}{p})$, 
	\begin{equation}\label{ClaimBMO}
		(L^p(\R^n), \text{BMO}(\R^n))_{\frac{s p}{n}, p} \hookrightarrow s^{-\frac{1}{p}}  (n- s p) L^{p^{*}, p}(\R^n).
	\end{equation}
	Assuming momentarily the validity of this claim, then \eqref{BBMFract} follows easily from  \eqref{ForClaim-1new} and \eqref{ForClaim0} (with $t = \frac{n}{p}$).
	
	It remains to show \eqref{ClaimBMO}. To proceed with, we will make use of the known estimate (cf. \cite[Corollary 3.3]{JawerthTorchinsky})
	$$
		K(u, f; L^p(\R^n), \text{BMO}(\R^n)) \approx \bigg(\int_0^{u^p} (f^{\#*}(v))^p \, dv \bigg)^{\frac{1}{p}}.
	$$
	Accordingly, by Fubini's theorem,
	\begin{align}
		\|f\|_{(L^p(\R^n), \text{BMO}(\R^n))_{\frac{s p}{n}, p}}^p &= \int_0^\infty (u^{-\frac{s p}{n}} K(u, f; L^p(\R^n), \text{BMO}(\R^n)) )^p \frac{du}{u} \nonumber \\
		& \approx \int_0^\infty u^{-\frac{s p}{n}} \int_0^{u} (f^{\#*}(v))^p \, dv \frac{du}{u} \label{10} \\
		& = \frac{n}{s p} \int_0^\infty  (v^{\frac{1}{p^{*}}}f^{\#*}(v))^p \frac{dv}{v}.  \nonumber
	\end{align}
	
%	As a combination of \eqref{10} with the celebrated Bennett--DeVore--Sharpley inequality [BDS, (3.3), p. 605],
%	$$
%		f^{**}(v)-f^*(v) \lesssim f^{\#*}(v)
%	$$
%	and (since $f^{**}(\infty) = 0$)
%	$$
%		f^{**}(u) = \int_u^\infty (f^{**}(v)- f^{*}(v)) \frac{dv}{v}
%	$$
%	(cf. [BS, (7.29), p. 384]), 
%	
	
	We can compare $f^{\#}$ and $f^{**}$ via the following weak-type estimate (cf. \cite[Proposition 8.10, p. 398]{BennettSharpley})
	$$
		(f-f_\infty)^{**}(u) \lesssim \int_u^\infty f^{\#*}(v) \frac{dv}{v} 
	$$
	where $f_\infty := \lim_{|Q| \to \infty} \frac{1}{|Q|} \int_Q f$. Note that $f_\infty = 0$ since $f \in L^p(\R^n)$. Accordingly, by Hardy's inequality (see, e.g., \cite[p. 196]{SteinWeiss}),
	\begin{align*}
		\|f\|_{L^{p^{*}, p}(\R^n)} &\leq \bigg(\int_0^\infty [u^{\frac{1}{p^*}} f^{**}(u)]^p \frac{du}{u} \bigg)^{\frac{1}{p}} \\
		& \lesssim  \bigg(\int_0^\infty \bigg(u^{\frac{1}{p^*}}  \int_u^\infty f^{\#*}(v) \frac{d v}{v} \bigg)^p \frac{du}{u} \bigg)^{\frac{1}{p}} \\
		& \lesssim p^* \bigg(\int_0^\infty u^{\frac{p}{p^*}} (f^{\#*}(u))^p \frac{du}{u} \bigg)^{\frac{1}{p}} \\
		& \approx s^{\frac{1}{p}} p^{*} \|f\|_{(L^p(\R^n), \text{BMO}(\R^n))_{\frac{s p}{n}, p}}
	\end{align*}
	where we have used \eqref{10} in the last step. 
	The proof of \eqref{ClaimBMO} is complete. 
	\\
	
	\textsc{Case 3:} Assume $t = \frac{n}{p}$ and $s \to t^{-}$. Let $t_0 \in (0, t)$ be fixed and choose $\theta \in (0, 1)$ such that $s = (1-\theta) t_0 + \theta t$. Then there exists a constant $C >0$, which is independent of $s$, such that
	\begin{equation}\label{3.1}
		\|f\|_{(\dot{H}^{t_0, p}(\R^n), \dot{H}^{t, p}(\R^n))_{\theta, p}} \leq C \,  \|f\|_{\dot{W}^{s, p}(\R^n), t}
	\end{equation}
	Indeed, since (see, e.g., \cite[Section 6.4, pp. 149--153]{BerghLofstrom})
	$$\dot{B}^{t_0}_{p, \min\{p, 2\}} (\R^n) = (L^p(\R^n), \dot{H}^{t, p}(\R^n))_{\frac{t_0}{t}, \min\{p, 2\}}$$
	and
	$$\dot{B}^{t_0}_{p, \min\{p, 2\}} (\R^n) \hookrightarrow \dot{H}^{t_0, p} (\R^n),$$  
	we have
	\begin{equation}\label{3.2}
	 ( (L^p(\R^n), \dot{H}^{t, p(}\R^n))_{\frac{t_0}{t}, \min\{p, 2\}}, \dot{H}^{t, p}(\R^n))_{\theta, p}	\hookrightarrow (\dot{H}^{t_0, p}(\R^n), \dot{H}^{t, p}(\R^n))_{\theta, p}
	\end{equation}
	with related embedding constant independent of $s$. Furthermore, in virtue of the sharp version of the reiteration formula given in Lemma \ref{LemmaKMX2}(ii) and \eqref{OrderedCouplesCommute} (and taking into account that $\theta \to 1^{-}$), the following holds
	$$
		(L^p(\R^{n}), \dot{H}^{t, p}(\R^n))_{1- (1-\theta)(1-\frac{t_0}{t}), p} \hookrightarrow ( (L^p(\R^n), \dot{H}^{t, p}(\R^n))_{\frac{t_0}{t}, \min\{p, 2\}}, \dot{H}^{t, p}(\R^n))_{\theta, p}
		$$
		uniformly with respect to $\theta$, or equivalently, by Lemma \ref{LemmaIntFractGS}, 
		\begin{equation}\label{3.3}
			\|f\|_{( (L^p(\R^n), \dot{H}^{t, p}(\R^n))_{\frac{t_0}{t}, \min\{p, 2\}}, \dot{H}^{t, p}(\R^n))_{\theta, p}} \lesssim \|f\|_{\dot{W}^{s, p}(\R^n), t}.
		\end{equation}
		Putting together \eqref{3.2} and \eqref{3.3}, the estimate \eqref{3.1} is achieved. 
		
		By the well-known fact (see, e.g., \cite[Theorem 1(i), Section 5.2.3, p. 242]{Triebel83}) that $(-\Delta)^{\frac{t_0}{2}}$ acts as an isomorphism from $\dot{H}^{t_0, p}(\R^n)$ onto $L^p(\R^n)$ and from $\dot{H}^{t, p}(\R^n)$ onto $\dot{H}^{t-t_0, p}(\R^n)$, we easily derive that
		$$
			K(u, f; \dot{H}^{t_0, p}(\R^n), \dot{H}^{t, p}(\R^n)) \approx K(u, (-\Delta)^{\frac{t_0}{2}} f; L^p(\R^n), \dot{H}^{t-t_0, p}(\R^n))
		$$
		and thus
		$$
			\|f\|_{(\dot{H}^{t_0, p}(\R^n), \dot{H}^{t, p}(\R^n))_{\theta, p}}  \approx \|(-\Delta)^{\frac{t_0}{2}} f\|_{(L^p(\R^n), \dot{H}^{t-t_0}_p(\R^n))_{\theta, p}}
		$$
		where the hidden constants are independent of $\theta$. According to Lemma \ref{LemmaIntFractGS}, the last estimate can be equivalently rewritten as
		\begin{equation}\label{3.4}
			\|f\|_{(\dot{H}^{t_0, p}(\R^n), \dot{H}^{t, p}(\R^n))_{\theta, p}}  \approx \|(-\Delta)^{\frac{t_0}{2}} f\|_{\dot{W}^{\theta (t-t_0), p}(\R^n), t-t_0}.
		\end{equation}
		Since $0 < \theta (t-t_0) < t-t_0 < t = \frac{n}{p}$, we can apply Case 1 so that (take into account that $f \in W^{s, p}(\R^n)$ and, in particular, $f \in H^{t_0, p}(\R^n)$ since $t_0 < s$)
		\begin{equation}\label{3.5}
			\|(-\Delta)^{\frac{t_0}{2}} f \|_{L^{\frac{n p}{n-(s-t_0)p} , p}(\R^n)} \lesssim (1-\theta)^{\frac{1}{p}}  \|(-\Delta)^{\frac{t_0}{2}} f\|_{\dot{W}^{\theta (t-t_0), p}(\R^n), t-t_0}.
		\end{equation}
		
		As usual, we denote by $I_\alpha, \, \alpha \in (0, n),$ the Riesz potential operator
		$$
			(I_\alpha f)(x) :=  c_{n, \alpha} \int_{\R^n} \frac{f(y)}{|x-y|^{n-\alpha}} \, dy, \qquad x \in \R^n.
		$$  		
		By well-known mapping properties of $I_\alpha$ (cf. \cite[Chapter V, Section 1.2]{Stein})
		\begin{equation}\label{3.6}
			I_{t_0}:  L^{\frac{n p}{n-(s-t_0)p}, p}(\R^n) \to L^{p^{*}, p}(\R^n).
		\end{equation}
		However, our method requires to control the norm of this operator with respect to $s$. Accordingly, to make the exposition self-contained, we provide below a detailed proof of \eqref{3.6} using standard techniques. It follows from O'Neil's inequality \cite[Lemma 1.5]{Oneil} that
		$$
			(I_{t_0} f)^*(u) \lesssim \int_u^\infty f^{**}(v) v^{\frac{t_0}{n}-1} \, dv
		$$
		and thus, by Hardy's inequalities (cf. \cite[p. 196]{SteinWeiss}),
		\begin{align*}
			\|I_{t_0} f\|_{L^{p^{*}, p}(\R^n)} & \lesssim \bigg(\int_0^\infty  u^{\frac{p}{p^*}}  \bigg(\int_u^\infty f^{**}(v) v^{\frac{t_0}{n}-1} \, dv \bigg)^{p} \frac{du}{u}  \bigg)^{\frac{1}{p}} \\
			& \lesssim p^{*} \bigg(\int_0^\infty \big(u^{\frac{n-(s-t_0) p}{n p}} f^{**}(u) \big)^p   \frac{du}{u}  \bigg)^{\frac{1}{p}} \\
			& \lesssim \frac{p^{*}}{1-\frac{1}{p} + \frac{s-t_0}{n}} \bigg(\int_0^\infty (u^{\frac{n-(s-t_0) p}{n p}} f^*(u))^p \frac{du}{u}  \bigg)^{\frac{1}{p}} \\
			& \approx \frac{1}{n- s p} \|f\|_{L^{\frac{n p}{n-(s-t_0) p}, p}(\R^n)}
		\end{align*}
		where we have used $s \to t^{-} = \big(\frac{n}{p} \big)^{-}$ in the last step. This proves (cf. \eqref{3.6})
		\begin{equation}\label{3.7}
			\|I_{t_0}\|_{L^{\frac{n p}{n-(s-t_0) p}, p}(\R^n) \to L^{p^{*}, p}(\R^n)} \lesssim (n- s p)^{-1}
		\end{equation}
		uniformly with respect to $s \to t^{-}$.
		
		Since $I_{t_0} (-\Delta)^{\frac{t_0}{2}} f = f$, in light of \eqref{3.7}, \eqref{3.5}, \eqref{3.4} and  \eqref{3.1} we get
		\begin{align*}
			\|f\|_{L^{p^{*}, p}(\R^n)} &\lesssim (n- s p)^{-1} \| (-\Delta)^{\frac{t_0}{2}} f \|_{L^{\frac{n p}{n-(s-t_0) p}, p}(\R^n)} \\
			& \lesssim  (n- s p)^{\frac{1}{p}-1}  \|(-\Delta)^{\frac{t_0}{2}} f\|_{\dot{W}^{\theta (t-t_0), p}(\R^n), t-t_0} \\
			& \approx (n- s p)^{\frac{1}{p}-1} \|f\|_{(\dot{H}^{t_0, p}(\R^n), \dot{H}^{t, p}(\R^n))_{\theta, p}} \\
			& \lesssim (n- s p)^{\frac{1}{p}-1} \|f\|_{\dot{W}^{s, p}(\R^n), t}.
		\end{align*}
		The proof is finished.
\end{proof}

As an immediate consequence of Theorem \ref{ThmBBMt} and the well-known inequality between Lorentz spaces (cf. \cite[p. 192]{SteinWeiss})
$$
	\Big(\frac{q}{p} \Big)^{\frac{1}{q}} \|f\|_{L^{p, q}(\R^n)} \leq \Big(\frac{r}{p} \Big)^{\frac{1}{r}} \|f\|_{L^{p, r}(\R^n)}, \qquad r < q,
$$
we obtain the following

\begin{corollary}
	Let $0 < s < t \leq 1, 1 < p < \frac{n}{s}$ and $p^{*} = \frac{n p}{n-s p}$. Assume $f \in W^{s, p}(\R^n)$. Then there exists $C = C(n, p, t) > 0$ such that
	\begin{equation*}
	\|f\|^p_{L^{p^{*}}(\R^n)} \leq C \, s (t-s) (n- s p)^{1-p} \|f\|_{\dot{W}^{s, p}(\R^n), t}^p.
	\end{equation*}
\end{corollary}

\section*{Appendix B: Sharp relationships between Sobolev--Triebel--Lizorkin
spaces\label{SectionAppendix}}

The methodology proposed in \cite{Braz} relies on sharp comparison estimates
between $\dot{W}^{s,p}({\mathbb{R}}^{n}),\dot{F}_{p,2}^{s}({\mathbb{R}}^{n})$
and $\dot{F}_{p,p}^{s}({\mathbb{R}}^{n})$ using tools from harmonic analysis.
We now provide an approach to these results relying on the interpolation
techniques presented in Section \ref{SectionMethod}.

Let $s\in(0,1)$ and $p\in(1,\infty)$. Recall that (cf. Lemma
\ref{TableCoincidences})
\begin{equation}
\dot{W}^{s,p}({\mathbb{R}}^{n})=\dot{F}_{p,p}^{s}({\mathbb{R}}^{n}%
),\label{GTL}%
\end{equation}%
\[
\dot{F}_{p,2}^{s}({\mathbb{R}}^{n})\hookrightarrow\dot{W}^{s,p}({\mathbb{R}%
}^{n}),\qquad p\geq2,
\]
and
\[
\dot{W}^{s,p}({\mathbb{R}}^{n})\hookrightarrow\dot{F}_{p,2}^{s}({\mathbb{R}%
}^{n}),\qquad p\leq2.
\]
The sharp behavior of the related equivalence constants in these embeddings
plays a key role in the Bourgain--Brezis--Mironescu method. In particular, note
that the family of semi-norms $\{\Vert\cdot\Vert_{\dot{F}_{p,p}^{s}%
({\mathbb{R}}^{n})}:s\in(0,1)\}$ converges to $\Vert\cdot\Vert_{\dot{F}%
_{p,p}^{1}({\mathbb{R}}^{n})}$ as $s\rightarrow1^{-}$, however this is not the
case for $\{\Vert\cdot\Vert_{\dot{W}^{s,p}({\mathbb{R}}^{n})}:s\in(0,1)\}$
(cf. \eqref{Constancy}). In particular, the hidden equivalence constants in
\eqref{GTL} must exhibit a certain blow up as $s$ approaches the limiting
values $0$ and $1$. The precise answer is contained in the following

\begin{theorem}
\label{TheoremBSY1} Let $0 < s < t < \infty$ and $p \in(1, \infty)$. Then
there exists a positive constant $C$, which is independent of $s$, such that

\begin{enumerate}
\item[(a)]
\begin{align}
C^{-1}\bigg(\frac{1}{s^{\frac{1}{\max\{p,2\}}}}+\frac{1}{(t-s)^{\frac{1}%
{\max\{p,2\}}}}\bigg)\Vert f\Vert_{\dot{F}_{p,p}^{s}({\mathbb{R}}^{n})}  &
\leq\Vert f\Vert_{\dot{W}^{s,p}({\mathbb{R}}^{n}),t}\nonumber\\
&  \hspace{-4cm}\leq C\bigg(\frac{1}{s^{\frac{1}{\min\{p,2\}}}}+\frac
{1}{(t-s)^{\frac{1}{\min\{p,2\}}}}\bigg)\Vert f\Vert_{\dot{F}_{p,p}%
^{s}({\mathbb{R}}^{n})}, \label{BSYEstim}%
\end{align}

\item[(b)]
\[
C\bigg(\frac{1}{s^{\frac{1}{p}}}+\frac{1}{(t-s)^{\frac{1}{p}}}\bigg)\Vert f\Vert_{\dot
{F}_{p,2}^{s}({\mathbb{R}}^{n})}\leq\Vert f\Vert_{\dot{W}^{s,p}({\mathbb{R}%
}^{n}),t},\qquad p\leq2,
\]

\item[(c)]
\[
\Vert f\Vert_{\dot{W}^{s,p}({\mathbb{R}}^{n}),t}\leq C\bigg(\frac{1}{s^{\frac{1}{p}}%
}+\frac{1}{(t-s)^{\frac{1}{p}}}\bigg)\Vert f\Vert_{\dot{F}_{p,2}^{s}({\mathbb{R}}%
^{n})},\qquad p\geq2.
\]

\end{enumerate}
\end{theorem}

\begin{remark}
The previous result can be understood as a generalization to higher order
smoothness of \cite[Theorems 1.2 and 1.5]{Braz} (which corresponds to the
classical case $t=1$, cf. \eqref{FractModuInt}).
\end{remark}

\begin{proof}
[Proof of Theorem \ref{TheoremBSY1}](a): By Lemma \ref{LemmaIntFractGS}, for
every $s\in(0,t)$,
\begin{equation}
\Vert f\Vert_{\dot{W}^{s,p}({\mathbb{R}}^{n}),t}\approx\Vert f\Vert
_{(L^{p}({\mathbb{R}}^{n}),\dot{H}^{t,p}({\mathbb{R}}^{n}))_{\frac{s}{t},p}}.
\label{ThmProof1}%
\end{equation}
The interpolation space given in the right-hand side of the previous estimate
can be explicitly computed via the retraction method. Specifically,
$(L^{p}({\mathbb{R}}^{n}),\dot{H}^{t,p}({\mathbb{R}}^{n}))$ can be identified
with the vector-valued pair $(L^{p}({\mathbb{R}}^{n};\ell_{2}({{\mathbb{Z}}%
})),\ell^{p}({\mathbb{R}}^{n};\ell_{2}^{t}({{\mathbb{Z}}})))$ and thus, by
Lemma \ref{LemmaInterpolationLpVector},
\begin{equation}
\Vert(f_{j})\Vert_{(L^{p}({\mathbb{R}}^{n};\ell_{2}({{\mathbb{Z}}}%
)),L^{p}({\mathbb{R}}^{n};\ell_{2}^{t}({{\mathbb{Z}}})))_{\frac{s}{t},p}%
}\approx\Vert(f_{j})\Vert_{L^{p}({\mathbb{R}}^{n};(\ell_{2}({{\mathbb{Z}}%
}),\ell_{2}^{t}({{\mathbb{Z}}}))_{\frac{s}{t},p})}. \label{ThmProof1new}%
\end{equation}
Furthermore, according to Lemma \ref{Lemma4},
\begin{align}
\bigg(\frac{1}{s^{\frac{1}{\max\{p,2\}}}}+\frac{1}{(t-s)^{\frac{1}%
{\max\{p,2\}}}}\bigg)\Vert\xi\Vert_{\ell_{p}^{s}({{\mathbb{Z}}})}  &
\lesssim\Vert\xi\Vert_{(\ell_{2}({{\mathbb{Z}}}),\ell_{2}^{t}({{\mathbb{Z}}%
}))_{\frac{s}{t},p}}\nonumber\\
&  \hspace{-4cm}\lesssim\bigg(\frac{1}{s^{\frac{1}{\min\{p,2\}}}}+\frac
{1}{(t-s)^{\frac{1}{\min\{p,2\}}}}\bigg)\Vert\xi\Vert_{\ell_{p}^{s}%
({{\mathbb{Z}}})}. \label{ThmProof1new2}%
\end{align}
As a combination of \eqref{ThmProof1new} and \eqref{ThmProof1new2},
\begin{align}
\bigg(\frac{1}{s^{\frac{1}{\max\{p,2\}}}}+\frac{1}{(t-s)^{\frac{1}%
{\max\{p,2\}}}}\bigg)\Vert(f_{j})\Vert_{L^{p}({\mathbb{R}}^{n};\ell_{p}%
^{s}({{\mathbb{Z}}}))}  &  \lesssim\Vert(f_{j})\Vert_{(L^{p}({\mathbb{R}}%
^{n};\ell_{2}({{\mathbb{Z}}})),L^{p}({\mathbb{R}}^{n};\ell_{2}^{t}%
({{\mathbb{Z}}})))_{\frac{s}{t},p}}\nonumber\\
&  \hspace{-5cm}\lesssim\bigg(\frac{1}{s^{\frac{1}{\min\{p,2\}}}}+\frac
{1}{(t-s)^{\frac{1}{\min\{p,2\}}}}\bigg)\Vert(f_{j})\Vert_{L^{p}({\mathbb{R}%
}^{n};\ell_{p}^{s}({{\mathbb{Z}}}))}. \label{ThmProof1new3}%
\end{align}
By the well-known fact that $\dot{F}_{p,p}^{s}({\mathbb{R}}^{n})$ is a retract
of $L^{p}({\mathbb{R}}^{n};\ell_{p}^{s}({{\mathbb{Z}}}))$, we can put together
\eqref{ThmProof1new3} and \eqref{ThmProof1} to arrive at \eqref{BSYEstim}.

(b): Let $p\leq2$. Applying Lemmas \ref{Lemma4} and \ref{LemmaExtrapolSharp}%
(i),
\begin{equation}
\Vert\xi\Vert_{\ell_{2}^{s}({{\mathbb{Z}}})}\approx(s(t-s))^{\frac{1}{2}}%
\Vert\xi\Vert_{(\ell_{2}({{\mathbb{Z}}}),\ell_{2}^{t}({{\mathbb{Z}}}%
))_{\frac{s}{t},2}}\lesssim(s(t-s))^{\frac{1}{p}}\Vert\xi\Vert_{(\ell
_{2}({{\mathbb{Z}}}),\ell_{2}^{t}({{\mathbb{Z}}}))_{\frac{s}{t},p}}.
\label{ThmProof1new4}%
\end{equation}
Taking into account that $\dot{F}_{p,2}^{s}({\mathbb{R}}^{n})$ is a retract of
$L^{p}({\mathbb{R}}^{n};\ell_{2}^{s}({{\mathbb{Z}}}))$, it follows now from
\eqref{ThmProof1}, \eqref{ThmProof1new} and \eqref{ThmProof1new4} that
\[
(s(t-s))^{-\frac{1}{p}}\Vert f\Vert_{\dot{F}_{p,2}^{s}({\mathbb{R}}^{n}%
)}\lesssim\Vert f\Vert_{\dot{W}^{s,p}({\mathbb{R}}^{n}),t}.
\]

The proof of (c) follows similar ideas as those given in (b).
\end{proof}

\end{document}